\documentclass[a4paper,11pt]{article}

\usepackage{amsmath,amsthm,amssymb}
\usepackage{graphicx,subcaption,tikz,tkz-graph}
\usepackage{xspace,paralist}							                              
\usepackage[shortlabels]{enumitem}	
\usepackage{hyperref}       
\hypersetup{
    colorlinks=true,
    citecolor = blue,
    linkcolor=black,
    filecolor=magenta,
    urlcolor=cyan,
    bookmarks=true,
}

\newtheorem{theorem}{Theorem}
\newtheorem{proposition}{Proposition}
\newtheorem{lemma}{Lemma}
\newtheorem{claim}{Claim}
\newtheorem{claimx}{Claim}

\newcommand\emitem[1]{\item{\itshape #1}\\}
	
\newcommand \e {\hfill {\tiny $\blacksquare$}}
\newcommand{\NP}{{\sf NP}}
\renewcommand{\P}{{\sf P}}
\newcommand{\col}[1]{\textsc{$#1$-Colouring\xspace}}
\newcommand{\cn}{\textsc{Colouring}\xspace}
\newcommand{\lc}{\textsc{List Colouring}\xspace}
\DeclareMathOperator{\cw}{cw}

\setlist[1]{itemsep=1em}

\title{{\bf Colouring Square-Free Graphs\\ without Long Induced Paths}\thanks{An extended abstract of this paper appeared in the proceedings of STACS 2018~\cite{GHP18}.}}

\author{
Serge Gaspers\thanks{School of Computer Science and Engineering, UNSW Sydney, Sydney 2052, Australia.} \thanks{Decision Sciences Group, Data61, CSIRO, Sydney 2052, Australia.}
\and
Shenwei Huang\thanks{Department of Mathematics, Wilfrid Laurier University, Waterloo N2L3C5, Canada.}
\and
Dani\"{e}l Paulusma\thanks{Department of Computer Science, Durham University, Durham DH13LE, UK.}
}

\date{May 17, 2018}

\begin{document}

\maketitle

\begin{abstract} The complexity of {\sc Colouring} is fully understood for $H$-free graphs, but there are still major complexity gaps if two induced subgraphs $H_1$ and $H_2$ are forbidden. 
Let $H_1$ be the $s$-vertex cycle $C_s$ and $H_2$ be the $t$-vertex path $P_t$.  
We show that {\sc Colouring} is polynomial-time solvable for $s=4$ and $t\leq 6$, strengthening several known results. 
Our main approach is to initiate a study into the boundedness of the clique-width of atoms (graphs with no clique cutset) of a hereditary graph class. We first show that the classifications of boundedness of clique-width of $H$-free graphs and $H$-free atoms coincide. We then show that this is not the case if two graphs are forbidden: we prove that $(C_4,P_6)$-free atoms have clique-width at most~18. Our key proof ingredients are a divide-and-conquer approach for bounding the clique-width of a subclass of $C_4$-free graphs
and the construction of a new bound on the clique-width for (general) graphs in terms of the clique-width of recursively defined 
subgraphs induced by homogeneous pairs and triples of sets.
As a complementary result we prove that {\sc Colouring} is \NP-complete for $s=4$ and $t\geq 9$, 
which is the first hardness result on {\sc Colouring} for $(C_4,P_t)$-free graphs.
Combining our new results with known results leads to an almost complete dichotomy for \cn restricted to $(C_s,P_t)$-free graphs. 
\end{abstract}

\section{Introduction}

Graph colouring has been a popular and extensively studied concept in computer science and mathematics 
since its introduction as a map colouring problem more than 150 years ago due to its many application areas 
crossing disciplinary boundaries and to its use as a benchmark problem in research into computational hardness.
The corresponding decision problem, {\sc Colouring}, is to decide, for a given graph $G$ and integer~$k$, 
if $G$ admits a {\it $k$-colouring}, that is, a mapping $c:V(G)\to \{1,\dots,k\}$ such that $c(u)\neq c(v)$ whenever $uv\in E(G)$. 
Unless $\P=\NP$, it is not possible to solve {\sc Colouring} in polynomial time for general graphs, 
not even if the number of colours is limited to 3~\cite{Lo73}.
To get a better understanding of the borderline between tractable and intractable instances of {\sc Colouring}, it is natural
to restrict the input to some special graph class.
Hereditary graph classes, which are classes of graphs closed under vertex deletion, 
provide a unified framework for a large collection of well-known graph classes.
It is readily seen that a graph class is hereditary if and only if it can be characterized by a unique set 
${\cal H}$ of minimal forbidden induced subgraphs. Graphs with no induced subgraph isomorphic to a graph in a set 
${\cal H}$ are called {\it ${\cal H}$-free}.  

Over the years, the study of \cn for hereditary graph classes  
has evolved into a deep 
area of research in theoretical computer science and discrete mathematics (see, for example,~\cite{BLS99, Golu04, JT95, MR02}).
One of the best-known results 
is the classical result of  Gr\"otschel, Lov\'asz, and Schrijver~\cite{GLS84},
who showed that {\sc Colouring} is polynomial-time solvable for perfect graphs. 
Faster, even linear-time, algorithms are known for subclasses of perfect graphs, such as chordal graphs,
bipartite graphs, interval graphs, and comparability graphs; see for example~\cite{Golu04}.
All these classes are characterized by {\it infinitely} many minimal forbidden induced subgraphs.

Kr\'al', Kratochv\'{\i}l, Tuza, and Woeginger~\cite{KKTW01} initiated a systematic study into the computational complexity of 
\cn restricted to hereditary graph classes characterized by a {\it finite} number of minimal forbidden induced subgraphs. In particular
they gave a complete classification of the complexity of \cn for the case where ${\cal H}$ consists of a single graph~$H$. 

\begin{theorem}[\cite{KKTW01}]\label{t-dicho}
If $H$ is an induced subgraph of $P_4$ or of $P_1+ P_3$, then
{\sc colouring} restricted to $H$-free graphs is polynomial-time solvable, otherwise it is \NP-complete. 
\end{theorem}
\autoref{t-dicho} led to two natural directions for further research:
\begin{enumerate}
\item Is it possible to obtain a dichotomy for {\sc Colouring} on $H$-free graphs if the number of colours $k$ is fixed 
(that is, $k$ no longer belongs to the input)?\\[-8mm] 
\item Is it possible to obtain a dichotomy for {\sc Colouring} on $\mathcal{H}$-free graphs if ${\cal H}$ has size~2?
\end{enumerate}
We briefly discuss known results for both directions below and refer to~\cite{GJPS17} for a detailed survey. 
Let $C_s$ and $P_t$ denote the cycle on $s$ vertices and path on $t$ vertices, respectively. 
We start with the first question. If $k$ is fixed, then
we denote the problem by $k$-{\sc Colouring}. 
It is known that for every 
$k\geq 3$, the {\sc $k$-Colouring} problem on $H$-free graphs is \NP-complete whenever $H$ contains a cycle~\cite{EHK98} or 
an induced claw~\cite{Ho81,LG83}. Therefore, only the case when $H$ is a disjoint union of paths remains.
In particular, the situation where $H=P_t$ has been thoroughly studied.
On the positive side, $3$-{\sc Colouring} on $P_7$-free graphs~\cite{BCMSZ}, $4$-{\sc Colouring} on $P_6$-free 
graphs~\cite{CSZ,CSZb} and 
$k$-{\sc Colouring} on $P_5$-free graphs for any $k\ge 1$~\cite{HKLSS10} are polynomial-time solvable.
On the negative side, Huang~\cite{Hu16} proved \NP-completeness for ($k=5$, $t=6$) and for ($k=4$, $t=7$). The case
($k=3$, $t\geq 8$) remains open, although some partial results are known~\cite{CS}.

In this paper we focus on the second question, that is, we restrict the input of {\sc Colouring} to $\mathcal{H}$-free graphs for 
${\cal H}=\{H_1,H_2\}$. 
For two graphs $G$ and $H$, we write $G+H=(V(G)\cup V(H),E(G)\cup E(H))$ for the disjoint union of two vertex-disjoint graphs $G$ and $H$, and $rG$ for the disjoint union of $r$ copies of $G$.
As a starting point, Kr\'al', Kratochv\'{\i}l, Tuza, and Woeginger~\cite{KKTW01} 
identified the following three main sources of \NP-completeness:
\begin{itemize}
\item both $H_1$ and $H_2$ contain a claw;\\[-8mm] 
\item both $H_1$ and $H_2$ contain a cycle; and\\[-8mm] 
\item both $H_1$ and $H_2$ contain an induced subgraph from the set $\{4P_1,2P_1+P_2,2P_2\}$.
\end{itemize}
They also showed additional \NP-completeness results by mixing the three types.
Since then numerous 
papers~\cite{BDJP16,BGPS12a,CH,DDP17,DGP14,DP18,HH17,HL,Hu16,KMP,KKTW01,LM15,Ma13,Ma,ML17,Sc05}
have been devoted to this problem,  but despite all these efforts the complexity classification for {\sc Colouring} on $(H_1,H_2)$-free graphs 
is still far from complete, and even dealing with specific pairs $(H_1,H_2)$ may require substantial work.

One of the ``mixed''
results obtained in~\cite{KKTW01} is that
{\sc Colouring} is \NP-complete for $(C_s,H)$-free graphs when $s\geq 5$ and $H\in \{4P_1,2P_1+P_2,2P_2\}$. 
This, together with the well-known result that {\sc Colouring} can be solved in linear time for $P_4$-free graphs 
(see also \autoref{t-dicho}) implies the following dichotomy.

\begin{theorem}[\cite{KKTW01}]\label{t-5}
Let $s\geq 5$ be a fixed integer. Then {\sc Colouring} for $(C_s,P_t)$-free graphs 
is polynomial-time solvable when $t\leq 4$ and \NP-complete when $t\geq 5$.
\end{theorem}
\autoref{t-5}
raises the natural question: 
what is the complexity of {\sc Colouring} on $(C_s,P_t)$-free graphs when $s\in \{3,4\}$?

For $s=3$, Huang, Johnson and Paulusma~\cite{HJP14} proved that {\sc 4-Colouring}, and thus {\sc Colouring},
is \NP-complete for $(C_3,P_{22})$-free graphs.
A result of Brandst{\"a}dt, Klembt and Mahfud \cite{BKM06} implies that {\sc Colouring} is polynomial-time solvable
for $(C_3,P_{6})$-free graphs.

For $s=4$, it is only known  that {\sc Colouring} is polynomial-time solvable for $(C_4,P_5)$-free graphs~\cite{Ma13}. 
This is unless we fix the number of colours: for every $k\geq 1$ and $t\geq 1$, it is known that {\sc $k$-Colouring} is polynomial-time solvable for 
$(K_{r,r},P_t)$-free graphs for every $s\geq 1$, and thus for $(C_4,P_t)$-free graphs~\cite{GPS14} (take $r=2$).
The underlying reason for this is a result of Atminas, Lozin and Razgon~\cite{ALR12}, who proved that every $P_t$-free graph either has small treewidth or contains a large biclique $K_{s,s}$ as a subgraph. Then Ramsey arguments can be used but only if the number of colours~$k$ is fixed. 
The result for $s=4$ and fixed $k$ is in contrast to the result of~\cite{HJP14} that for all $k\geq 4$ and $s\geq 5$, 
there exists a constant~$t^s_k$ such that  $k$-{\sc Colouring} is \NP-complete even for
$(C_3,C_5,\ldots,C_{s},P_{t^s_k})$-free graphs. 
	
\subsection*{Our Main Results} 
We show, 
in \autoref{sec:poly}, 
that \cn is polynomial-time solvable for $(C_4,P_6)$-free graphs.
The class of $(C_4,P_6)$-free graphs generalizes the classes of split graphs (or equivalently, $(C_4,C_5,2P_2)$-free graphs)
and pseudosplit graphs (or equivalently, $(C_4,2P_2)$-free graphs).
The case of $(C_4,P_6)$-free graphs was explicitly mentioned as a natural case to consider in~\cite{GJPS17}.
Our result unifies several previous results on colouring $(C_4,P_t)$-free graphs, 
namely:
the polynomial-time solvability of \cn for $(C_4,P_5)$-free graphs \cite{Ma13};
the polynomial-time solvability of \col{k} for $(C_4,P_6)$-free graphs for every $k\geq 1$~\cite{GPS14};
and the recent $3/2$-approximation algorithm for {\sc Colouring} for $(C_4,P_6)$-free graphs~\cite{GH17}.
It also complements a recent result of Karthick and Maffray~\cite{KM} who gave
tight linear upper bounds of the chromatic number of a $(C_4,P_6)$-free graph in terms of its clique number and maximum degree that strengthen a similar bound given in~\cite{GH17}.

It was not previously known 
if there exists an integer $t$ such that \cn is \NP-complete for $(C_4,P_{t})$-free graphs.
In \autoref{sec:hard} we complement our positive result of~\autoref{sec:poly} by giving an affirmative answer to this question:
already the value $t=9$ makes the problem \NP-complete.
 
\subsection*{Our Methodology}
The general research goal of our paper is to increase, in a systematic way, 
our insights in the computational hardness of {\sc Colouring} by developing new techniques. 
In particular we aim to narrow the complexity gaps between the hard and easy cases.
Clique-width is a well-known width parameter and having bounded clique-width is often the underlying reason
for a large collection of \NP-complete problems, including {\sc Colouring}, to become polynomial-time solvable on a special graph class; 
this follows from results of~\cite{CMR00,EGW01,KR03b,OS06,Ra07}.
For this reason we want to use clique-width to solve {\sc Colouring} for $(C_4,P_6)$-free graphs, 
However, the class of $(C_4,P_6)$-free graphs has unbounded clique-width, as it contains the class of {\it split graphs}, or equivalently,
$(C_4,C_5,2P_2)$-free graphs, which may have arbitrarily large clique-width~\cite{MR99}. 

To overcome this obstacle we first preprocess the $(C_4,P_6)$-free input graph. An {\em atom} is a graph with no clique cutset.
Clique cutsets were introduced by Dirac~\cite{Di61}, who proved that every chordal graph is either complete or has a clique cutset.
Later, decomposition into atoms became a very general tool for solving combinatorial problems on chordal graphs and other hereditary graph classes, such as those that forbid some Truemper configuration~\cite{BPV17}.
For instance, {\sc Colouring} and also other problems, such as {\sc Independent Set} and {\sc Clique}, are polynomial-time solvable 
on a hereditary graph class~${\cal G}$ if they are so on the atoms of ${\cal G}$~\cite{Ta85}.
Hence, we may restrict ourselves to the subclass of $(C_4,P_6)$-free atoms in order to solve {\sc Colouring} for $(C_4,P_6)$-free graphs.

Adler et al.~\cite{ALMRTV17} proved that (diamond, even-hole)-free atoms have unbounded clique-with.
However, so far, (un)boundedness of the cliquewidth of atoms in special graph classes has not been well studied. 
It is known that a class of $H$-free graphs has bounded clique-width if and only if 
$H$ is an induced subgraph of $P_4$ (see~\cite{DP16}). 
As a start of a more systematic study, we show in~\autoref{s-atoms} that the same result holds fo atoms: a class of $H$-free atoms has bounded clique-width if and only if 
$H$ is an induced subgraph of $P_4$.
In contrast, we observe that,
although split graphs have unbounded clique-width~\cite{MR99},
split atoms are cliques~\cite{Di61} and thus have clique-width at most~2. 
Recall that split graphs are characterized by three forbidden induced subgraphs. This yields the natural question  whether one can
prove the same result for a graph class characterized by two forbidden induces subgraphs. 
In this paper we give an {\it affirmative} answer to this question by showing that the class of $(C_4,P_6)$-free atoms has bounded clique-width. 
As mentioned, this immediately yields a polynomial-time algorithm for \cn on $(C_4,P_6)$-free graphs,

In order to prove that $(C_4,P_6)$-free atoms have bounded clique-width, we further develop the approach of~\cite{GH17}
used to bound the chromatic number of $(C_4,P_6)$-free graphs as a linear function of their maximum clique size
and to obtain a $3/2$-approximation algorithm for {\sc Colouring} for $(C_4,P_6)$-free graphs. 
The approach of~\cite{GH17} is based on a decomposition theorem for $(C_4,P_6)$-free atoms.
We derive a new variant of this decomposition theorem for so-called strong atoms, 
which are atoms that contain no universal vertices and no pairs of twin vertices.
We use this decomposition to prove that  $(C_4,P_6)$-free strong atoms have bounded clique-width.
To obtain this result we also apply a divide-and-conquer approach for bounding the clique-width of a subclass of $C_4$-free graphs. 
As another novel element of our proof, we show a new bound on the clique-width for (general) graphs 
in terms of the clique-width of recursively defined subgraphs induced by homogeneous triples and pairs of sets.
Our techniques may be of independent interest and can possibly be used to prove polynomial-time solvability of \cn
on other graph classes.

\medskip
\noindent
{\bf Remark}. The {\sc Independent Set} problem is to decide if a given graph $G$ has an independent set of at least~$k$ vertices 
for some given integer~$k$. Brandst{\"a}dt and Ho\`{a}ng~\cite{BH07} proved that {\sc Independent Set} is polynomial-time solvable for
$(C_4,P_6)$-free graphs. As mentioned, just as for {\sc Colouring}, it suffices to consider only the atoms of a hereditary graph class 
in order to solve {\sc Independent Set}~\cite{Ta85}. Brandst{\"a}dt and Ho\`{a}ng followed this approach. 
Although we will use one of their structural results as lemmas, their method does not yield a polynomial-time algorithm for {\sc Colouring} on $(C_4,P_6)$-free graphs.

\section{Preliminaries}\label{sec:pre}

Let $G=(V,E)$ be a graph. 
For $S\subseteq V$, the subgraph \emph{induced} by $S$, is denoted by $G[S]=(S,\{uv\; |\; u,v\in S\})$.
The {\it complement} of $G$ is the graph $\overline{G}$ with vertex set $V$ and edge set $\{uv\; |\; uv\notin E\}$.
A clique $K\subseteq  V$ is a \emph{clique cutset} if $G-K$ has more connected components than $G$. 
If $G$ has no clique cutsets, then $G$ is called an \emph{atom}.
The {\it edge subdivision} of an edge $uv\in E$ removes $uv$ from $G$ and replaces it by a new vertex~$w$ and two new edges $uw$ and $wv$.

The \emph{neighbourhood} of a vertex $v$ is denoted by $N(v)=\{u\; |\; uv\in E\}$ and its degree by $d(v)=|N(v)|$.
For a set $X\subseteq V$, we write $N(X)=\bigcup_{v\in X}N(v)\setminus X$.
For $x\in V$ and $S\subseteq V$, we let $N_S(x)$ be the set of neighbours of $x$ that are in $S$, that is, $N_S(x)=N_G(x)\cap S$. A subset $D\subseteq V$ is a \emph{dominating set} of $G$ if every vertex not in $D$ has a neighbour in $D$.
A vertex~$u$ is \emph{universal} in $G$ if it is adjacent to all other vertices, that is, $\{u\}$ is a dominating set of $G$.

For $X,Y\subseteq V$, we say that $X$ is \emph{complete} (resp. \emph{anti-complete}) to $Y$
if every vertex in $X$ is adjacent (resp. non-adjacent) to every vertex in $Y$.
Let $u,v\in V$ be two distinct vertices. We say that a vertex $x\notin \{u,v\}$ \emph{distinguishes} $u$ and $v$
if $x$ is adjacent to exactly one of $u$ and $v$.  A set $H\subseteq V$ is a \emph{homogeneous set} 
if no vertex in $V\setminus H$ can distinguish two vertices in $H$. A homogeneous set $H$ is \emph{proper}
if $1<|H|<|V|$. A graph is \emph{prime} if it contains no proper homogeneous set.

We say that $u$ and $v$ are \emph{(true) twins} if $u$ and $v$ are adjacent
and have the same set of neighbours in $V\setminus \{u,v\}$. Note that the binary relation
of being twins is an equivalence relation on $V$, and so $V$ can be partitioned into equivalence 
classes $T_1,\ldots, T_r$ of twins. The \emph{skeleton} of $G$ is the subgraph induced by 
a set of $r$ vertices, one from each of $T_1,\ldots,T_r$.
A \emph{blow-up} of $G$ is a graph $G'$ obtained by replacing each vertex $u\in V$ by 
a clique $K_u$  of size at least~$1$, such that two distinct cliques $K_u$ and $K_v$ are complete in $G'$ if $u$ and $v$ are adjacent in $G$,
and anti-complete otherwise. Since each equivalence class of twins is a clique and any two equivalence classes
are either complete or anti-complete, every graph is a blow-up of its skeleton.

Let $\{H_1,\ldots, H_p\}$ be a set of $p$ graphs for some integer $p\geq 1$. We say that $G$ is {\it $(H_1,\ldots,H_p)$-free}
if $G$ contains no induced subgraph isomorphic to $H_i$ for some $1\leq i\leq p$. If $p=1$, we may write that $G$ is $H_1$-free instead. If $V$ can be partitioned into a clique $C$ and an independent set $I$, then $G$ is a {\it split graph}.

The \emph{clique-width} of a graph~$G$, denoted by~$\cw(G)$, 
is the minimum number of labels required to construct~$G$ using the following four operations:
\begin{itemize}
\item $i(v)$: create a new graph consisting of a single vertex $v$ with label $i$;\\[-8mm] 
\item $G_1\oplus G_2$: take the disjoint union of two labelled graphs $G_1$ and $G_2$;\\[-8mm] 
\item $\eta_{i,j}$: join each vertex with label $i$ to each vertex with label $j$ (for $i\neq j$);\\[-8mm] 
\item $\rho_{i\rightarrow j}$: rename label $i$ to $j$.
\end{itemize}
A \emph{clique-width expression} of $G$ is an algebraic expression 
that describes how $G$ can be recursively constructed using these operations.
An \emph{$\ell$-expression} of $G$ is a clique-width expression using at most $\ell$ distinct labels.
For instance, $\eta_{3,2}(3(d)\oplus \rho_{3\rightarrow 2}(\rho_{2\rightarrow 1}(\eta_{3,2}(3(c)\oplus \eta_{2,1}(2(b)\oplus 1(a))))))$
is a $3$-expression for the path on vertices $a,b,c,d$ in that order.
A class of graphs~${\cal G}$ has {\em bounded} clique-width if there is a constant~$c$ such that the clique-width 
of every graph in~${\cal G}$ is at most~$c$,and {\em unbounded} otherwise.

Clique-width is of fundamental importance in computer science, since 
every problem expressible in monadic second-order logic using
quantifiers over vertex subsets but not over edge subsets becomes polynomial-time solvable for graphs of bounded 
clique-width~\cite{CMR00}. Although this meta-theorem does not directly apply to \cn, a result of
Espelage, Gurski and Wanke~\cite{EGW01} (see also~\cite{KR03b,Ra07}) 
combined with the polynomial-time approximation algorithm of Oum and Seymour~\cite{OS06}
for finding an $\ell$-expression of a graph,
showed that \cn can be added to the list of such problems.

\begin{theorem}[\cite{EGW01,OS06}]\label{thm:kr03}
\cn can be solved in polynomial time for graphs of bounded clique-width.
\end{theorem}

\section{Atoms and Clique-Width}\label{s-atoms}

Recall that a graph is an atom if it contains no clique cutset. It is a natural question whether the clique-width of a graph class
of unbounded clique-width becomes bounded after restricting to the atoms of the class.
For classes of $H$-free graphs we note that this is not the case though. 
In order to explain this we need the notion of a wall; see \autoref{f-walls} 
for three examples (for a formal definition we refer to, for example,~\cite{Chuzhoy15}). 
A \emph{$k$-subdivided wall} is a graph obtained from a wall after subdividing each edge 
exactly~$k$ times for some constant $k\geq 0$. The following lemma is well known.

\begin{lemma}[\cite{LR06}]\label{l-walls}
For every constant $k\geq 0$, the class of $k$-subdivided walls has unbounded clique-width.
\end{lemma}

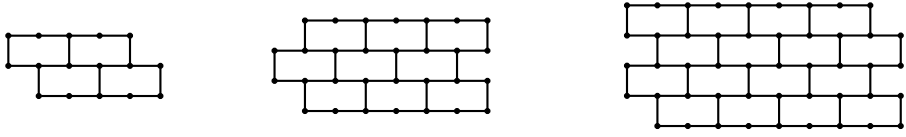
\begin{figure}
\begin{center}
\begin{minipage}{0.2\textwidth}
\centering
\begin{tikzpicture}[scale=0.4, every node/.style={scale=0.3}]
\GraphInit[vstyle=Simple]
\SetVertexSimple[MinSize=6pt]
\Vertex[x=1,y=0]{v10}
\Vertex[x=2,y=0]{v20}
\Vertex[x=3,y=0]{v30}
\Vertex[x=4,y=0]{v40}
\Vertex[x=5,y=0]{v50}

\Vertex[x=0,y=1]{v01}
\Vertex[x=1,y=1]{v11}
\Vertex[x=2,y=1]{v21}
\Vertex[x=3,y=1]{v31}
\Vertex[x=4,y=1]{v41}
\Vertex[x=5,y=1]{v51}

\Vertex[x=0,y=2]{v02}
\Vertex[x=1,y=2]{v12}
\Vertex[x=2,y=2]{v22}
\Vertex[x=3,y=2]{v32}
\Vertex[x=4,y=2]{v42}

\Edges(    v10,v20,v30,v40,v50)
\Edges(v01,v11,v21,v31,v41,v51)
\Edges(v02,v12,v22,v32,v42)

\Edge(v01)(v02)

\Edge(v10)(v11)

\Edge(v21)(v22)

\Edge(v30)(v31)

\Edge(v41)(v42)

\Edge(v50)(v51)

\end{tikzpicture}
\end{minipage}
\begin{minipage}{0.3\textwidth}
\centering
\begin{tikzpicture}[scale=0.4, every node/.style={scale=0.3}]
\GraphInit[vstyle=Simple]
\SetVertexSimple[MinSize=6pt]
\Vertex[x=1,y=0]{v10}
\Vertex[x=2,y=0]{v20}
\Vertex[x=3,y=0]{v30}
\Vertex[x=4,y=0]{v40}
\Vertex[x=5,y=0]{v50}
\Vertex[x=6,y=0]{v60}
\Vertex[x=7,y=0]{v70}

\Vertex[x=0,y=1]{v01}
\Vertex[x=1,y=1]{v11}
\Vertex[x=2,y=1]{v21}
\Vertex[x=3,y=1]{v31}
\Vertex[x=4,y=1]{v41}
\Vertex[x=5,y=1]{v51}
\Vertex[x=6,y=1]{v61}
\Vertex[x=7,y=1]{v71}

\Vertex[x=0,y=2]{v02}
\Vertex[x=1,y=2]{v12}
\Vertex[x=2,y=2]{v22}
\Vertex[x=3,y=2]{v32}
\Vertex[x=4,y=2]{v42}
\Vertex[x=5,y=2]{v52}
\Vertex[x=6,y=2]{v62}
\Vertex[x=7,y=2]{v72}

\Vertex[x=1,y=3]{v13}
\Vertex[x=2,y=3]{v23}
\Vertex[x=3,y=3]{v33}
\Vertex[x=4,y=3]{v43}
\Vertex[x=5,y=3]{v53}
\Vertex[x=6,y=3]{v63}
\Vertex[x=7,y=3]{v73}

\Edges(    v10,v20,v30,v40,v50,v60,v70)
\Edges(v01,v11,v21,v31,v41,v51,v61,v71)
\Edges(v02,v12,v22,v32,v42,v52,v62,v72)
\Edges(    v13,v23,v33,v43,v53,v63,v73)

\Edge(v01)(v02)

\Edge(v10)(v11)
\Edge(v12)(v13)

\Edge(v21)(v22)

\Edge(v30)(v31)
\Edge(v32)(v33)

\Edge(v41)(v42)

\Edge(v50)(v51)
\Edge(v52)(v53)

\Edge(v61)(v62)

\Edge(v70)(v71)
\Edge(v72)(v73)
\end{tikzpicture}
\end{minipage}
\begin{minipage}{0.35\textwidth}
\centering
\begin{tikzpicture}[scale=0.4, every node/.style={scale=0.3}]
\GraphInit[vstyle=Simple]
\SetVertexSimple[MinSize=6pt]
\Vertex[x=1,y=0]{v10}
\Vertex[x=2,y=0]{v20}
\Vertex[x=3,y=0]{v30}
\Vertex[x=4,y=0]{v40}
\Vertex[x=5,y=0]{v50}
\Vertex[x=6,y=0]{v60}
\Vertex[x=7,y=0]{v70}
\Vertex[x=8,y=0]{v80}
\Vertex[x=9,y=0]{v90}

\Vertex[x=0,y=1]{v01}
\Vertex[x=1,y=1]{v11}
\Vertex[x=2,y=1]{v21}
\Vertex[x=3,y=1]{v31}
\Vertex[x=4,y=1]{v41}
\Vertex[x=5,y=1]{v51}
\Vertex[x=6,y=1]{v61}
\Vertex[x=7,y=1]{v71}
\Vertex[x=8,y=1]{v81}
\Vertex[x=9,y=1]{v91}

\Vertex[x=0,y=2]{v02}
\Vertex[x=1,y=2]{v12}
\Vertex[x=2,y=2]{v22}
\Vertex[x=3,y=2]{v32}
\Vertex[x=4,y=2]{v42}
\Vertex[x=5,y=2]{v52}
\Vertex[x=6,y=2]{v62}
\Vertex[x=7,y=2]{v72}
\Vertex[x=8,y=2]{v82}
\Vertex[x=9,y=2]{v92}

\Vertex[x=0,y=3]{v03}
\Vertex[x=1,y=3]{v13}
\Vertex[x=2,y=3]{v23}
\Vertex[x=3,y=3]{v33}
\Vertex[x=4,y=3]{v43}
\Vertex[x=5,y=3]{v53}
\Vertex[x=6,y=3]{v63}
\Vertex[x=7,y=3]{v73}
\Vertex[x=8,y=3]{v83}
\Vertex[x=9,y=3]{v93}

\Vertex[x=0,y=4]{v04}
\Vertex[x=1,y=4]{v14}
\Vertex[x=2,y=4]{v24}
\Vertex[x=3,y=4]{v34}
\Vertex[x=4,y=4]{v44}
\Vertex[x=5,y=4]{v54}
\Vertex[x=6,y=4]{v64}
\Vertex[x=7,y=4]{v74}
\Vertex[x=8,y=4]{v84}

\Edges(    v10,v20,v30,v40,v50,v60,v70,v80,v90)
\Edges(v01,v11,v21,v31,v41,v51,v61,v71,v81,v91)
\Edges(v02,v12,v22,v32,v42,v52,v62,v72,v82,v92)
\Edges(v03,v13,v23,v33,v43,v53,v63,v73,v83,v93)
\Edges(v04,v14,v24,v34,v44,v54,v64,v74,v84)

\Edge(v01)(v02)
\Edge(v03)(v04)

\Edge(v10)(v11)
\Edge(v12)(v13)

\Edge(v21)(v22)
\Edge(v23)(v24)

\Edge(v30)(v31)
\Edge(v32)(v33)

\Edge(v41)(v42)
\Edge(v43)(v44)

\Edge(v50)(v51)
\Edge(v52)(v53)

\Edge(v61)(v62)
\Edge(v63)(v64)

\Edge(v70)(v71)
\Edge(v72)(v73)

\Edge(v81)(v82)
\Edge(v83)(v84)

\Edge(v90)(v91)
\Edge(v92)(v93)
\end{tikzpicture}
\end{minipage}
\caption{Walls of height 2, 3, and 4, respectively~\cite{DP16}.}\label{f-walls}
\end{center}
\end{figure}

We also need the following lemma.

\begin{lemma}\label{l-wallatom}
For every constant $k\geq 0$, every $k$-subdivided wall and every complement of a $k$-subdivided wall is an atom.
\end{lemma}

\begin{proof}
Let $k\geq 0$. Let $W$ be a $k$-subdivided wall. As $W$ is $C_3$-free, a largest clique has size~2. It is readily seen
that $W$ contains no set of at most two vertices that disconnect $W$. 

Now consider the  complement $\overline{W}$ of $W$.
For contradiction, assume that $\overline{W}$ is not an atom. Then $\overline{W}$ has a clique cutset~$K$.
Let $A$ and $B$ be two connected components of $\overline{W}-K$.
If $A$ and $B$ both have at least two vertices $a_1,a_2$ and $b_1,b_2$, respectively, then $W[\{a_1,a_2,b_1,b_2\}]$ contains 
a $C_4$, which is not possible. Hence, one of $A,B$, say $A$, only contains one vertex~$a$.
As the neighbourhood of $a$ in $\overline{W}$ is a clique, the non-neighbourhood of $a$ in $W$ is an independent set.
However, no vertex in $W$ has this property.
\end{proof}

Recall that a class of $H$-free graphs has bounded clique-width if and only if $H$ is an induced subgraph of $P_4$ (see~\cite{DP16}). 
We show that the same classification holds for $H$-free atoms.

\begin{proposition}\label{p-p4}
Let~$H$ be a graph. The class of $H$-free atoms has bounded clique-width if and only if~$H$ is an induced subgraph of~$P_4$.
\end{proposition}

\begin{proof}
If~$H$ is an induced subgraph of~$P_4$, then the class of $H$-free graphs, which contains all $H$-free atoms,
has clique-width at most~$2$~\cite{CO00}.

Now suppose that $H$ is not an induced subgraph of $P_4$.
For every $k\geq 0$, every $k$-subdivided wall is an atom by Lemma~\ref{l-wallatom}.
First suppose that $H$ contains a cycle. Then the class of $k$-subdivided walls 
is contained in the class of $H$-free atoms for some appropriate value of $k$. 
Hence, the class of $H$-free atoms has unbounded clique-width due to \autoref{l-walls}.

Now suppose that $H$ does not contain a cycle. Hence $H$ is a forest.
As $H$ is not an induced subgraph of $P_4$, we find that $H$ must contain an induced $3P_1$ or an induced $2P_2$. 
Let ${\cal G}$ be the class of $H$-free atoms, and
let $\overline{\cal G}$ be the class that consists of the complements of $H$-free atoms.
As every wall is $(C_3,C_4)$-free, the complement of every wall is $(3P_1,2P_2)$-free. By Lemma~\ref{l-wallatom},
the complement of every wall is an atom as well. Hence, $\overline{\cal G}$ contains all complements of walls.
It is well known that complementing all graphs in a class of unbounded clique-width results in another class of unbounded 
clique-width~\cite{KLM09}. Hence, complements of walls have unbounded clique-width due to Lemma~\ref{l-walls}.
This means that $\overline{\cal G}$, and thus ${\cal G}$, has unbounded clique-width.
\end{proof}

In contrast to \autoref{p-p4}, we recall that there exist classes of $(H_1,H_2,H_3)$-free graphs of unbounded clique-width whose atoms have bounded clique-width. Namely, the class of split graphs, or equivalently, 
the class of $(C_4,C_5,2P_2)$-free graphs, has  unbounded clique-width~\cite{MR99}, whereas split atoms are cliques and thus have clique-width at most~2. 
In the next section, we will prove that there exist even classes of $(H_1,H_2)$-free graphs with this property by showing that the property holds even for the class of $(C_4,P_6)$-free graphs.

\section{The Polynomial-Time Result}\label{sec:poly}

In this section, we will prove our main result.

\begin{theorem}\label{thm:p6c4}
\cn is polynomial-time solvable for $(C_4,P_6)$-free graphs. 
\end{theorem}

The main ingredient for proving \autoref{thm:p6c4} is a new structural property of $(C_4,P_6)$-free atoms, 
which asserts that $(C_4,P_6)$-free atoms have bounded clique-width.
 The following result is due to Tarjan.
 
\begin{theorem}[\cite{Ta85}]\label{t-tarjan}
If {\sc Colouring} is polynomial-time solvable on atoms in an hereditary class~${\cal G}$, then it is polynomial-time solvable on all graphs in~${\cal G}$.
\end{theorem}

As the class of $(C_4,P_6)$-free graphs is hereditary, we can apply \autoref{t-tarjan} and may restrict ourselves to $(C_4,P_6)$-free atoms. Then, due to \autoref{thm:kr03}, it suffices to show the following 
result in order to prove~\autoref{thm:p6c4}.

\begin{theorem}\label{thm:p6c4atom}
The class of $(C_4,P_6)$-free atoms has bounded clique-width.
More precisely, every $(C_4,P_6)$-free atom has clique-width at most $18$.
\end{theorem}

Note that $(C_4,P_6)$-free atoms are an example of a class of $(H_1,H_2)$-free graphs of unbounded clique-width, 
whose atoms have bounded clique-width.

The remainder of the section is organised as follows. In \autoref{subsec:cw}, we present
the key tools on clique-width that play an important role in the proof of \autoref{thm:p6c4atom}.
In~\autoref {subsec:C5}, we list structural properties around a 5-cycle in a $(C_4,P_6)$-free graph
that are frequently used in later proofs.  We then present the proof of~\autoref{thm:p6c4atom}
in~\autoref{subsec:main proof}.

\subsection{Key Tools for Clique-Width}\label{subsec:cw}

Let $G=(V,E)$ be a graph and $H$ be a proper homogeneous set in $G$.
Then $V\setminus H$ is partitioned into two subsets $N$ and $M$ where
$N$ is complete to $H$ and $M$ is anti-complete to $H$.
Let $h\in H$ be an arbitrary vertex and $G_h = G - (H\setminus \{h\})$.
We say that $H$ and $G_h$ are \emph{factors} of $G$ with respect to $H$.
Suppose that $\tau$ is an $\ell_1$-expression for $G_h$ using labels $1,\ldots,\ell_1$
and $\sigma$ is an $\ell_2$-expression for $H$ using labels $1,\ldots,\ell_2$. 
Then substituting $i(h)$ in $\tau$ with $\rho_{1\rightarrow i}\ldots \rho_{\ell_2\rightarrow i}\sigma$
results in an $\ell$-expression for $G$ where $\ell=\max\{\ell_1,\ell_2\}$. 
Moreover, all vertices in $H$ have the same label in this $\ell$-expression for $G$. 

\begin{lemma}[\cite{CO00}]\label{lem:prime}
The clique-width of any graph $G$ is the maximum clique-width of
any prime induced subgraph of $G$.
\end{lemma}

A bipartite graph  is a \emph{chain} graph
if it is $2P_2$-free. A \emph{co-bipartite chain graph} is the complement
of a bipartite chain graph. Let $G$ be a (not necessarily bipartite) graph such that $V(G)$ is partitioned into
two subsets $A$ and $B$. We say that an $\ell$-expression for $G$ is \emph{nice}
if all vertices in $A$ end up with the same label $i$ and all vertices in $B$ end up with the same label $j$
with $i\neq j$. It is well-known that any co-bipartite chain graph whose vertex set is partitioned into two cliques
has a nice $4$-expression (see Appendix~\ref{a-a} for a proof).

\begin{lemma}[Folklore]\label{lem:co-bipartite chain graph}
There is a nice $4$-expression for any co-bipartite chain graph.
\end{lemma}

We now use a divide-and-conquer approach to show that a special graph class has a nice 4-expression.
This plays a crucial role in our proof of the main theorem (\autoref{thm:p6c4atom}).

\begin{lemma}\label{lem:recursion}
A $C_4$-free graph $G$ has a nice $4$-expression if $V(G)$ can be partitioned into two (possibly empty) subsets $A$ and $B$ that
satisfy the following conditions:
\begin{enumerate}[(i)]
\item $G[A]$ is a clique;\label{item:clique1}\\[-22pt]
\item $G[B]$ is $P_4$-free;\\[-22pt]
\item no vertex in $A$ has two non-adjacent neighbours in $B$;\label{item:c4}\\[-22pt]
\item there is no induced $P_4$ in $G$ that starts with a vertex in $A$ followed by three vertices in $B$.\label{item:P4}
\end{enumerate} 
\end{lemma}
\begin{proof}
We use induction on $|B|$.
If $B$ contains at most one vertex, then $G$ is a co-bipartite chain graph and the lemma follows from \autoref{lem:co-bipartite chain graph}. 
Assume that $B$ contains at least two vertices.
Since $G[B]$ is $P_4$-free, either $B$ or $\overline{B}$ is disconnected \cite{Se74}.
Suppose first that $B$ is disconnected. Then  $B$ can be partitioned into two nonempty subsets $B_1$ and $B_2$
that are anti-complete to each other. 
Let $A_1=N(B_1)\cap A$ and $A_2=A\setminus A_1$.
Then $G[A_i\cup B_i]$, with partition $(A_i,B_i)$, satisfies 
conditions~\ref{item:clique1}--\ref{item:P4} for $i=1,2$. 
Note also that, by \ref{item:c4}, $A_1$ is anti-complete to $B_2$ and $A_2$ is anti-complete to $B_1$.
By the inductive hypothesis there is a nice $4$-expression $\tau_i$ for $G[A_i\cup B_i]$ in which 
all vertices in $A_i$ and $B_i$ have labels $2$ and $4$, respectively.
Now $\rho_{1\rightarrow 2}(\eta_{1,2}(\tau_1\oplus \rho_{2\rightarrow 1}\tau_2))$
is a  nice $4$-expression for $G$.
 
Suppose now that $\overline{B}$ is disconnected.  This means that $B$ can be partitioned into two subsets $B_1$ and $B_2$
that are complete to each other. Since $G$ is $C_4$-free, either $B_1$ or $B_2$ is a clique. Without loss generality, we may assume that
$B_1$ is a clique. Moreover, we choose the partition $(B_1,B_2)$ such that $B_1$ is maximal, so $B_1\neq \emptyset$. 
Then every vertex in $B_2$ is not adjacent to some vertex in $B_2$, for otherwise we could have moved such a vertex to $B_1$.
If $B_2=\emptyset$ then $G$ is a co-bipartite chain graph and so the lemma follows from  \autoref{lem:co-bipartite chain graph}. 
Therefore, we assume that $B_2\neq \emptyset$.
Let $A_1=N(B_1)\cap A$ and $A_2=A\setminus A_1$. 
Note that $A_2$ is anti-complete to $B_1$.

We claim that $N(B_2)\cap A$ is complete to $B_1$. 
Suppose, by contradiction, that $a\in N(B_2)\cap A$ and $b_1\in B_1$ are not adjacent. 
By definition, $a$ has a neighbour $b\in B_2$.  Recall that $b$ is not adjacent to some vertex $b'\in B_2$.
Now $a,b,b_1,b'$ induces either a $P_4$ or a $C_4$, depending on whether $a$ and $b'$ are adjacent. 
This contradicts \ref{item:P4} or the $C_4$-freeness of $G$.
This proves the claim. 
Since  $N(B_2)\cap A$ is complete to $B_1$, we find that $A_2$ is anti-complete to $B_2$ and $N(B_2)\cap A=N(B_2)\cap A_1$
(see \autoref{fig:connectB}).

Note that $G[(A_1\cap N(B_2))\cup B_2]$, with the partition $(A_1\cap N(B_2),B_2)$ satisfies 
conditions~\ref{item:clique1}--\ref{item:P4}.
By the inductive hypothesis there is a nice $4$-expression $\tau$ for $G[(A_1\cap N(B_2))\cup B_2]$ 
in which all vertices in $A\cap N(B_2)=A_1\cap N(B_2)$ and $B_2$ have labels $2$ and $4$, respectively.
As $A_1$ and $B_1$ are cliques and $G$ is $C_4$-free, we find that $(A_1\setminus N(B_2), B_1)$ is a co-bipartite chain graph.
It then follows from \autoref{lem:co-bipartite chain graph} that there is a nice $4$-expression $\epsilon$ 
for it in which all vertices in $A_1\setminus N(B_2)$ and $B_1$ have labels $1$ and $3$, respectively.

We now are going to use the adjacency between the different sets as displayed in \autoref{fig:connectB}.
We first deduce that 
\[\sigma=\rho_{3\rightarrow 4}(\rho_{1\rightarrow 2}(\eta_{3,4}(\eta_{2,3}(\eta_{1,2}(\epsilon\oplus \tau)))))\]
is a nice $4$-expression for $G-A_2$.  
Let $\delta$ be a $2$-expression for $A_2$ in which all vertices in $A_2$ have label $1$.
Then $\rho_{1\rightarrow 2}(\eta_{1,2}(\delta\oplus \sigma))$ is a nice $4$-expression for $G$. This completes the proof.
\end{proof}

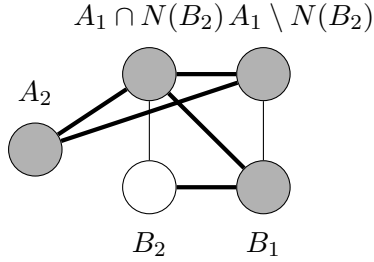
\begin{figure}[t]
\center
\begin{tikzpicture}[scale=0.5]
\tikzstyle{vertex}=[draw, circle, fill=white!100, minimum width=20pt,inner sep=2pt]
\tikzstyle{set}=[draw, circle, fill=black!30, minimum width=20pt, inner sep=2pt]

\node[set] (B1) at (3,0) {};
\node at (3,-1.5) {$B_1$};
\node[vertex] (B2) at (0,0) {};
\node at (0,-1.5) {$B_2$};
\node[set] (A1minusNB2) at (3,3) {};
\node at (4,4.5) {$A_1\setminus N(B_2)$};
\node[set] (A1intesectNB2) at (0,3) {};
\node at (0,4.5) {$A_1\cap N(B_2)$};
\node[set] (A2) at (-3,1) {};
\node at (-3,2.5) {$A_2$};

\draw[ultra thick]
(A2)--(A1minusNB2)--(A1intesectNB2)--(A2)
 (A1intesectNB2)--(B1)--(B2);
\draw
(A1intesectNB2)--(B2)
(A1minusNB2)--(B1);
\end{tikzpicture}
\caption{The case where $\overline{B}$ is disconnected. Shaded circles represent cliques.
A thick line between two sets represents that the two sets are complete;
a thin line means that the edges between the two sets are arbitrary, and no line means that the two sets are anti-complete.}\label{fig:connectB}
\end{figure} 

Let $G=(V,E)$ be a graph and $X$, $Y$, and $Z$ are three pairwise disjoint subsets of $V$.
We say that  $(X,Y,Z)$ is a {\em homogeneous triple}
if no vertex in $V\setminus (X\cup Y\cup Z)$ can distinguish any two vertices in $X$, $Y$ or $Z$.
A pair $(X,Y)$ of sets is a {\em homogeneous pair} if $(X,Y,\emptyset)$ is a homogeneous triple.
If both $X$ and $Y$ are cliques, then $(X,Y)$ is a {\em homogeneous pair of cliques}.
Note that homogeneous sets are special cases of homogeneous pairs and triples.
An $\ell$-expression for a homogeneous triple $(X,Y,Z)$ is {\em nice} 
if two vertices of $X\cup Y\cup Z$ have the same label if and only if they belong to the same set $X$, $Y$ or $Z$.
We establish a new bound on the clique-width of a graph~$G$ in terms of the number of pairwise disjoint homogenous pairs and triples of $G$. 

\begin{lemma}\label{lem:boundcw}
Let $G$ be a graph.
If $V(G)$ can be partitioned into a subset $V_0$,
with $|V_0|\ge 3$, and 
$p$ homogeneous pairs and $t$ homogeneous triples 
such that there is a 
nice $4$-expression for each homogeneous pair and a nice $6$-expression
for each homogeneous triple, then $\cw(G)\le |V_0|+2p+3t$.
\end{lemma}

\begin{proof}
We first construct the homogenous pairs and triples and the edges inside these pairs and triples one by one using nice $4$-expressions and nice $6$-expressions, respectively.
So we need at most four different labels for each homogenous pair and at most six different labels for each homogenous triple.
As soon as we have constructed a homogenous pair (triple) with its internal edges using a nice $4$-expression ($6$-expression), we introduce a new label for all vertices of each of its two (three) sets before considering the next homogenous pair or triple. We can do so, because all vertices of each set in a homogeneous pair received the same label by the definition of a nice $\ell$-expression for homogenous sets and triples.
Consequently, we may use the previous labels over and over again as auxiliary labels. 
Afterwards, we can view each set in a homogeneous pair or triple as a single vertex, each with its own unique label. 

So far we used at most $2p+3t+6$ different labels. By using the auxiliary labels as unique labels for the sets of the last pair or triple and by considering pairs before triples, we need in fact at most $2p+3t+3$ distinct labels if $t\geq 1$ and at most
$2p+2$ labels if $t=0$. We now assign a unique label to each vertex in~$V_0$ after first using all the remaining auxiliary labels.

So far we only constructed edges of $G$ that are within a homogenous pair or triple. From our labelling procedure and the definitions of homogenous pairs and triples it follows that
we can put in all the remaining edges of $G$ using only join and disjoint union operations. Hence, as $|V_0|\geq 3$, the total number of distinct labels is at most $|V_0|+2p+3t$.
\end{proof}

\subsection{Structure around a 5-Cycle}\label{subsec:C5}

Let $G=(V,E)$ be a graph and $H$ be an induced subgraph of $G$.
We partition $V\setminus V(H)$ into subsets with respect to $H$ as follows:
for any $X\subseteq V(H)$, we denote by $S(X)$ the set of vertices
in $V\setminus V(H)$ that have $X$ as their neighbourhood among $V(H)$, i.e.,
\[S(X)=\{v\in V\setminus V(H): N_{V(H)}(v)=X\}.\]
For $0\le j\le |V(H)|$, we denote by $S_j$ the set of vertices in $V\setminus V(H)$ that have exactly $j$
neighbours among $V(H)$. Note that $S_j=\bigcup_{X\subseteq V(H): |X|=j}S(X)$.
We say that a vertex in $S_j$ is a \emph{$j$-vertex}. 
Let $G$ be a $(C_4,P_6)$-free graph and $C=1,2,3,4,5$ be an induced $C_5$ in $G$.
We partition $V\setminus C$ with respect to $C$ as above. All indices below are modulo $5$.
Since $G$ is $C_4$-free, there is no vertex in $V\setminus C$ that is adjacent to vertices $i$ and $i+2$ but not to vertex $i+1$. In particular, $S(1,3)$, $S_4$, etc. are empty.
The following properties \ref{s5}-\ref{n22} of $S(X)$ were proved in \cite{HH17} using the fact that $G$ is $(C_4,P_6)$-free.
\begin{enumerate}[label=\bfseries (P\arabic*)]
\item {$S_5\cup S(i-1,i,i+1)$ is a clique}. \label{s5}

\item {$S(i)$ is complete to $S(i+2)$ and anti-complete to $S(i+1)$. 
Moreover, if neither $S(i)$ nor $S(i+2)$ are empty then both sets are cliques.}\label{s11}

\item  {$S(i,i+1)$ is complete to $S(i+1,i+2)$ and anti-complete to
$S(i+2,i+3)$. Moreover, if neither $S(i,i+1)$ nor $S(i+1,i+2)$ are empty then both sets are cliques.}\label{s22}

\item {$S(i-1,i,i+1)$ is anti-complete to $S(i+1,i+2,i+3)$.}\label{s33}

\item {$S(i)$ is anti-complete to $S(j,j+1)$ if $j\neq i+2$.
Moreover, if a vertex in $S(i+2,i+3)$ is not anti-complete to $S(i)$ then it is universal in $S(i+2,i+3)$.}\label{s12}

\item {$S(i)$ is anti-complete to $S(i+1,i+2,i+3)$.}\label{s13}

\item {$S(i-2,i+2)$ is anti-complete to $S(i-1,i,i+1)$.}\label{s23}

\item Either $S(i)$ or $S(i+1,i+2)$ is empty. By symmetry, either $S(i)$ or $S(i-1,i-2)$ is empty.\label{n12}

\item {At least one of $S(i-1,i)$, $S(i,i+1)$ and $S(i+2,i-2)$ is empty.}\label{n22}
\end{enumerate}

We now prove some further properties that are used in \autoref{lem:C5}.
\begin{enumerate}[label=\bfseries (P\arabic*)]
\setcounter{enumi}{9}
\emitem {For each connected component $A$ of $S(i)$, each vertex in $S(i-2,i-1,i)\cup S(i,i+1,i+2)$ is either complete or anti-complete to $A$.}
\label{item:p10}

{\it Proof.}
It suffices to prove the property for $i=1$. Suppose that some vertex $t\in S(4,5,1)\cup S(1,2,3)$ is neither complete nor anti-complete
to a connected component $A$ of $S(1)$. By symmetry, we may assume that $t\in S(4,5,1)$.
By the connectivity of $A$, there exists an edge $aa'$ in $A$ such that $t$ is
adjacent to $a$ but not to $a'$. Then $a',a,t,4,3,2$ induces a $P_6$, a contradiction. \e

\emitem {No vertex in $S_5$ can distinguish an edge between $S(i)$ and $S(i-2,i+2)$.}\label{item:s5nodis}

{\it Proof.}
It suffices to prove the property for $i=1$. Let $x\in S(1)$ and $y\in S(3,4)$ be adjacent. If  a vertex $u$ is adjacent to exactly one of $x$ and $y$,
then either $x,y,3,u$ or $x,y,u,1$ induces a $C_4$. \e

\emitem {If a vertex $x\in S(i-2,i+2)$ has a neighbour in $S(i-2,i-1,i)\cup S(i,i+1,i+2)$, then $x$ is complete to $S_5$.}\label{item:s25}

{\it Proof.} It suffices to prove the property for $i=1$. Suppose that $x$ is not adjacent to some $u\in S_5$.
Since $x$ has a neighbour $s\in S(1,2,3)\cup S(4,5,1)$, say $S(1,2,3)$, it follows that $s,u,4,x$ induces a $C_4$. \e

\emitem {Each vertex in $S(i-2,i+2)$ is anti-complete to either $S(i-2,i-1,i)$ or $S(i,i+1,i+2)$.}\label{item:p13}

{\it Proof.}
It suffices to prove the property for $i=1$.  Suppose that $x\in S(3,4)$ has a neighbour $s\in S(1,2,3)$ and $t\in S(4,5,1)$.
By \ref{s33}, $s$ and $t$ are not adjacent. Then $x,s,1,t$ induces a $C_4$. \e

\emitem {Each vertex in $S(i-1,i-2,i+2)$ and $S(i+1,i+2,i-2)$ is either complete or anti-complete 
to each connected component of $S(i-2,i+2)$.}\label{item:p14}

{\it Proof.}
It suffices to prove the property for $i=1$.  Suppose that $s\in S(2,3,4)\cup S(3,4,5)$ distinguishes an edge
$xy$ in $S(3,4)$, say $s$ is adjacent to $x$ but not to $y$. By symmetry, we  may assume that $s\in S(2,3,4)$.
Then $y,x,s,2,1,5$ induces a $P_6$. \e

\emitem {If both $S(i-1,i-2)$ and $S(i+1,i+2)$ are not empty, 
then each vertex in $S(i-1,i,i+1)$ is either complete or anti-complete to $S(i-1,i-2)\cup S(i+1,i+2)$.}\label{item:p15}

{\it Proof.}
It suffices to prove the property for $i=1$. Let $x\in S(2,3)$ and $y\in S(4,5)$ be two arbitrary vertices.
If $s\in S(5,1,2)$ distinguishes $x$ and $y$, say $s$ is adjacent to $x$ but not to $y$, 
then $1,s,x,3,4,y$  induces a $P_6$, since $y$ is not adjacent to $x$ by \ref{s22}. \e
\end{enumerate}

\subsection{Proof of \autoref{thm:p6c4atom}}\label{subsec:main proof}

In this section, we give a proof of \autoref{thm:p6c4atom}, which states that $(C_4,P_6)$-free graphs have clique-width 
at most~18.

A graph is \emph{chordal} if it does not contain any induced cycle of length at least~4.
The following structure of $(C_4,P_6)$-free graphs discovered by Brandst{\"a}dt and Ho\`{a}ng \cite{BH07} is of particular
importance in our proofs below.
\begin{theorem}[\cite{BH07}]\label{lem:P6C4atom}
Let $G$ be a $(C_4,P_6)$-free atom.  Then the following statements hold:
\begin{inparaenum}[(i)]
\item every induced $C_5$ is dominating;
\item if $G$ contains an induced $C_6$ which is not dominating, then $G$ is the join of
a blow-up of the Petersen graph (\autoref{fig:counterexample}) and a (possibly empty) clique.
\end{inparaenum} 
\end{theorem}

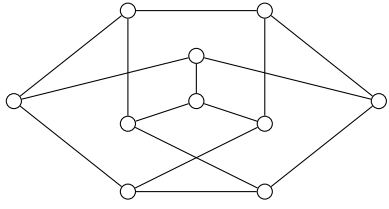
\begin{figure}[tb]
\centering
\begin{tikzpicture}[scale=0.6]
\tikzstyle{vertex}=[draw, circle, fill=white!100, minimum width=4pt,inner sep=2pt]

\node[vertex] (v1) at (-1.5,2) {};
\node[vertex] (v2) at (1.5,2) {};
\node[vertex] (v3) at (4,0) {};
\node[vertex] (v4) at (1.5,-2) {};
\node[vertex] (v5) at (-1.5,-2) {};
\node[vertex] (v6) at (-4,0) {};
\draw (v1)--(v2)--(v3)--(v4)--(v5)--(v6)--(v1);

\node[vertex] (v14) at (-1.5,-0.5) {};
\draw (v14)--(v1) (v14)--(v4);
\node[vertex] (v25) at (1.5,-0.5) {};
\draw (v25)--(v2) (v25)--(v5);
\node[vertex] (v36) at (0,1) {};
\draw (v36)--(v3) (v36)--(v6);
\node[vertex] (c) at (0,0) {};
\draw (c)--(v14) (c)--(v25) (c)--(v36);
\end{tikzpicture}
\caption{The Petersen graph.}
\label{fig:counterexample}
\end{figure}

We say that an atom is \emph{strong} if it has no pair of twin vertices or universal vertices.
Note that a pair of twin vertices and a universal vertex in a graph give rise to two special kinds of
proper homogeneous sets such that one of the factors decomposed by these homogeneous sets is a clique.
Therefore, removing  twin vertices and universal vertices does not change the clique-width of the graph by \autoref{lem:prime}.
So, to  prove \autoref{thm:p6c4atom} it suffices to prove the theorem for strong atoms.

We follow the approach of~\cite{GH17}. 
In \cite{GH17}, the  first and second author showed how to derive a useful decomposition
theorem for $(C_4,P_6)$-free atoms by eliminating a sequence $F_1$, $C_6$, $F_2$ and $C_5$ 
(see \autoref{fig:F1F2} for the graphs $F_1$ and $F_2$)
of induced subgraphs and then employing Dirac's classical theorem~\cite{Di61} on  chordal graphs.
Here we adopt the same strategy and show in \autoref{lem:F1}--\autoref{lem:C5} below that
if a $(C_4,P_6)$-free strong atom $G$ contains an induced $C_5$ or $C_6$, then it has clique-width at most 18.
The remaining case is therefore that $G$ is
a chordal atom, and so $G$ is a clique by Dirac's theorem~\cite{Di61}.
Since cliques have clique-width $2$, \autoref{thm:p6c4atom} follows.
It turns out that we can easily prove  \autoref{lem:F1} and \autoref{lem:C6} via the framework formulated
in \autoref{lem:boundcw} using the structure
of the graphs discovered in \cite{GH17}. The difficulty is, however, that we have to extend
the structural analysis in \cite{GH17} extensively for \autoref{lem:F2} and \autoref{lem:C5} and provide
new insights on bounding the clique-width of certain special graphs using divide-and-conquer (see \autoref{lem:recursion}).

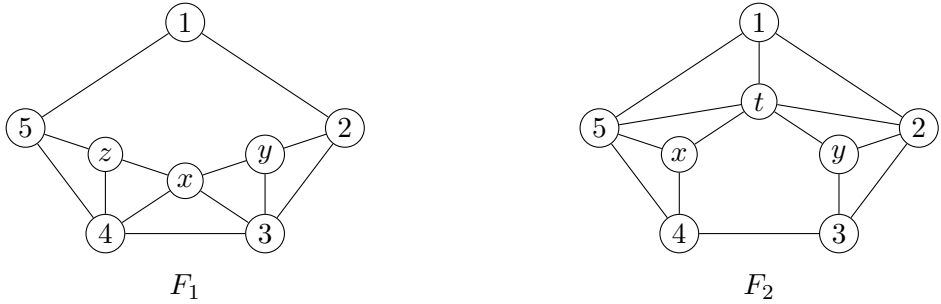
\begin{figure}[tb]
\centering
\begin{subfigure}{.5\textwidth}
\centering
\begin{tikzpicture}[scale=0.7]
\tikzstyle{vertex}=[draw, circle, fill=white!100, minimum width=4pt,inner sep=2pt]

\node [vertex] (v1) at (0,3) {$1$};
\node [vertex] (v2) at (3,1) {$2$};
\node [vertex] (v3) at (1.5,-1) {$3$};
\node [vertex] (v4) at (-1.5,-1) {$4$};
\node [vertex] (v5) at (-3,1) {$5$};
\draw 
(v1)--(v2)--(v3)--(v4)--(v5)--(v1);

\node[vertex] (y) at (1.5,0.5) {$y$};
\draw (v2)--(y) (v3)--(y);

\node[vertex] (z) at (-1.5,0.5) {$z$};
\draw (v4)--(z) (v5)--(z);

\node [vertex] (x) at (0,0) {$x$};
\draw (x)--(v3)   (x)--(v4)  (x)--(y)  (x)--(z);

\node at (0,-2) {$F_1$};
\end{tikzpicture}
\end{subfigure}%
\begin{subfigure}{.5\textwidth}
\centering
\begin{tikzpicture}[scale=0.7]
\tikzstyle{vertex}=[draw, circle, fill=white!100, minimum width=4pt,inner sep=2pt]

\node [vertex] (v1) at (0,3) {$1$};
\node [vertex] (v2) at (3,1) {$2$};
\node [vertex] (v3) at (1.5,-1) {$3$};
\node [vertex] (v4) at (-1.5,-1) {$4$};
\node [vertex] (v5) at (-3,1) {$5$};
\draw 
(v1)--(v2)--(v3)--(v4)--(v5)--(v1);

\node[vertex] (y) at (1.5,0.5) {$y$};
\draw (v2)--(y) (v3)--(y);

\node[vertex] (z) at (-1.5,0.5) {$x$};
\draw (v4)--(z) (v5)--(z);

\node [vertex] (t) at (0,1.5) {$t$};
\draw (t)--(v5)   (t)--(v1)  (t)--(v2)  (t)--(y)  (t)--(z);

\node at (0,-2) {$F_2$};
\end{tikzpicture}
\end{subfigure}
\caption{Two special graphs $F_1$ and $F_2$.}
\label{fig:F1F2}
\end{figure}

\begin{lemma}\label{lem:F1}
If a $(C_4,P_6)$-free strong atom $G$ contains an induced $F_1$, then $G$ has clique-width at most $13$.
\end{lemma}

\begin{proof}
Let $G$ be a $(C_4,P_6)$-free strong atom that contains an induced subgraph $H$ that is isomorphic to $F_1$
with $V(H)=\{1,2,3,4,5,x,y,z\}$ where $1,2,3,4,5,1$ induces the \emph{underlying} 5-cycle $C$
of $F_1$ and $x$ is adjacent to $3$ and $4$, $y$ is adjacent to $2$ and $3$, 
$z$ is adjacent to $4$ and $5$, and $x$ is adjacent to $y$ and $z$, see \autoref{fig:F1F2}. 
We partition $V(G)$ with respect to $C$.
We choose $H$ such that $|S_2|$ maximized.
Note that $x\in S(3,4)$, $y\in S(2,3)$ and $z\in S(4,5)$.
All indices below are modulo $5$. Since $G$ is an atom, it follows from
\autoref{lem:P6C4atom} that $S_0=\emptyset$. Moreover,
it follows immediately from the $(C_4,P_6)$-freeness of $G$
that~$V(G)=C\cup  S_1\cup \bigcup_{i=1}^{5}S(i,i+1)\cup \bigcup_{i=1}^{5}S(i-1,i,i+1)\cup S_5$.
If $S_5\neq \emptyset$, then  $G$ is a blow-up of the graph $F_3$ (see \autoref{fig:F3})  \cite{GH17}.
Since $G$ contains no twin vertices, $G$ is isomorphic to $F_3$ and so  has clique-width at most $9$.
If $S_5=\emptyset$ then $G$ has the structure prescribed in \autoref{fig:noS5}  \cite{GH17}.  
Note that $S(5,1,2)\cup \{1\}$ is a homogeneous clique in $G$ and so $S(5,1,2)=\emptyset$.
We partition $S(3,4)$ into two subsets $X=\{x\in S(3,4): x \textrm{ has a neighbour in }S(4,5,1)\}$
and $Y=S(3,4)\setminus X$. Note that $Y$ is anti-complete to $S(4,5,1)$. In addition,
$X$ is anti-complete to $S(1,2,3)$ since $G$ is $C_4$-free.
It is routine to check that each of $(X,S(4,5,1))$, $(Y,S(1,2,3))$, $(S(2,3),S(3,4,5))$ and $(S(4,5),S(2,3,4))$ is
a homogeneous pair of cliques in $G$.
Now $V(G)$ is partitioned into a subset $C$ of size $5$ and four homogeneous pairs of cliques. 
Since each pair of homogeneous cliques induces a co-bipartite chain graph and so has a nice 4-expression 
by \autoref{lem:co-bipartite chain graph}. So, $\cw(G)\le |V_0|+2\times 4=13$ by  \autoref{lem:boundcw}.
\end{proof}

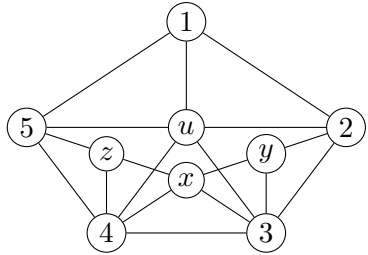
\begin{figure}[h!]
\centering
\begin{tikzpicture}[scale=0.7]
\tikzstyle{vertex}=[draw, circle, fill=white!100, minimum width=4pt,inner sep=2pt]

\node [vertex] (v1) at (0,3) {$1$};
\node [vertex] (v2) at (3,1) {$2$};
\node [vertex] (v3) at (1.5,-1) {$3$};
\node [vertex] (v4) at (-1.5,-1) {$4$};
\node [vertex] (v5) at (-3,1) {$5$};
\node[vertex] (u) at (0,1) {$u$};
\draw 
(v1)--(v2)--(v3)--(v4)--(v5)--(v1);

\node[vertex] (y) at (1.5,0.5) {$y$};
\draw (v2)--(y) (v3)--(y);

\node[vertex] (z) at (-1.5,0.5) {$z$};
\draw (v4)--(z) (v5)--(z);

\node [vertex] (x) at (0,0) {$x$};
\draw (x)--(v3)   (x)--(v4)  (x)--(y)  (x)--(z);

\draw
(u)--(v1) (u)--(v2) (u)--(v3) (u)--(v4) (u)--(v5);

\end{tikzpicture}
\caption{The graph $F_3$.}
\label{fig:F3}
\end{figure}

\begin{figure}[h!]
\center
\begin{tikzpicture}[scale=0.5]
\tikzstyle{vertex}=[draw, circle, fill=white!100, minimum width=4pt,inner sep=2pt]
\tikzstyle{set}=[draw, circle, minimum width=1pt, inner sep=1pt]

\node [set, blue] (v1) at (0,3) {$1$};
\node [set, blue] (v2) at (6,0.5) {$2$};
\node [set, blue] (v3) at (3,-3) {$3$};
\node [set,blue] (v4) at (-3,-3) {$4$};
\node [set,blue] (v5) at (-6,0.5) {$5$};

\node [set] (x) at (0,-1.5) {$S(3,4)$};
\node[set] (y) at (2.7,0) {$S(2,3)$};
\node[set] (z) at (-2.7,0) {$S(4,5)$};

\node[set] (t1) at (0,6) {$S(5,1,2)$};
\node[set] (t2) at (9, 0) {$S(1,2,3)$};
\node[set] (t5) at (-9,0) {$S(4,5,1)$};
\node[set] (t3) at (5,-5) {$S(2,3,4)$};
\node[set] (t4) at (-5,-5) {$S(3,4,5)$};

\draw [blue]
(v1)--(v2)--(v3)--(v4)--(v5)--(v1);
\draw[ultra thick]
(x)--(v3)
(x)--(v4)
(x)--(y)
(v2)--(y)
(v3)--(y)
(x)--(z)
(v4)--(z)
(v5)--(z)

(t1)--(v5) (t1)--(v1) (t1)--(v2)
(t2)--(v1) (t2)--(v2) (t2)--(v3)  (t2)--(y) (t2)--(t1)
(t5)--(v4) (t5)--(v5) (t5)--(v1) (t5)--(z) (t5)--(t1)
(t3)--(v2) (t3)--(v3) (t3)--(v4) (t3)--(y) (t3)--(x) (t3)--(t2)
(t4)--(v3) (t4)--(v4) (t4)--(v5) (t4)--(z) (t4)--(x) (t4)--(t5) (t4)--(t3);       

\draw
(t2)--(x)
(t5)--(x);
\draw[bend left]
(t3)--(z)
(t4)--(y);
\end{tikzpicture}
\caption{The structure of $G$. A thick line between two sets represents that the two sets are complete, and
a thin line represents that the edges between the two sets can be arbitrary.  Two sets are anti-complete
if there is no line between them.}\label{fig:noS5}
\end{figure}
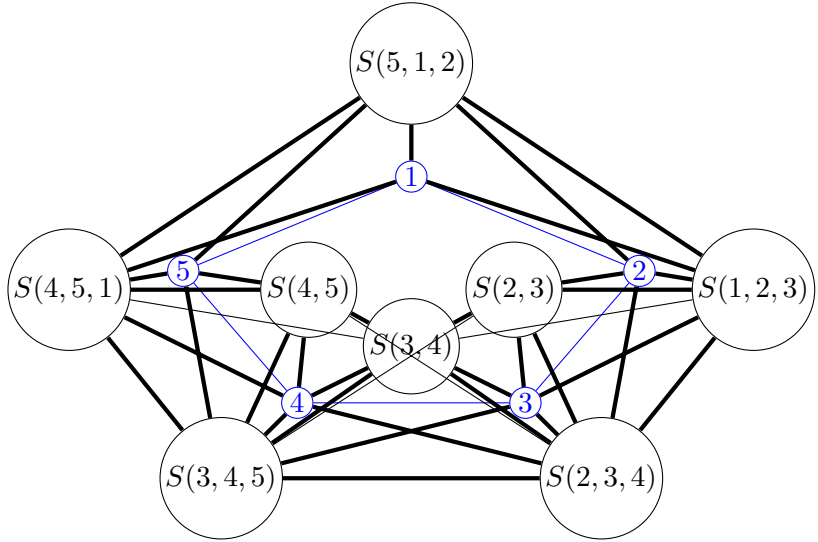

\begin{lemma}\label{lem:C6}
If a $(C_4,F_1,P_6)$-free strong atom $G$ contains an induced $C_6$, then $G$ has clique-width at most $13$.
\end{lemma}

\begin{proof}
Let $C=1,2,3,4,5,6,1$ be an induced six-cycle of $G$. We partition $V(G)$
with respect to $C$. If $C$ is not dominating, then it follows from \autoref{lem:P6C4atom} that $G$ is
the join of a blow-up of the Petersen graph and a (possibly empty) clique. 
Since $G$ has no twin vertices or universal vertices, $G$ is isomorphic to the Petersen graph
and so $G$ has clique-width at most $10$.
In the following, we assume that $C$ is dominating, i.e., $S_0= \emptyset$. 
It was shown in \cite{GH17} that $G$ is a blow-up of a graph of order at most $13$.
Therefore,  $G$ has clique-width at most $13$.
\end{proof}

\begin{lemma}\label{lem:F2}
If a $(C_4,C_6,F_1,P_6)$-free strong atom $G$ contains an induced $F_2$, then $G$ has clique-width at most $14$.
\end{lemma}

\begin{proof}
Let $G$ be a $(C_4,C_6,F_1,P_6)$-free strong atom that contains an induced subgraph 
$H$ that is isomorphic to $F_2$ with $V(H)=\{1,2,3,4,5,t,x,y\}$ such that $1,2,3,4,5,1$
induces the \emph{underlying} 5-cycle $C$, and $t$ is adjacent to $5$, $1$ and $2$, $x$ is adjacent to $4,5$ and $y$
is adjacent to $2$ and $3$. Moreover, $t$ is adjacent to both $x$ and $y$, see \autoref{fig:F1F2}.
We partition $V(G)$ with respect to $C$. 
We choose $H$ such that $C$ has $|S_2|$ maximized.
Note that $x\in S(4,5)$, $y\in S(2,3)$ and $t\in S(5,1,2)$.

The overall strategy is to first decompose $G$ into a subset $V_0$ of constant size and constant number of homogeneous pairs
of sets, and then finish off the proof via \autoref{lem:boundcw}
by showing that each homogeneous pair of sets has a nice 4-expression using \autoref{lem:recursion}.

We start with the decomposition.
Since $S(2,3)$ and $S(4,5)$ are not empty, it follows from \ref{n12} that $S_1 = S(2)\cup S(5)$.
If both $S(2)$ and $S(5)$ are not empty, say $u\in S(2)$ and $v\in S(5)$, 
then $u,2,3,4,5,v$ induces either a $P_6$ or a $C_6$, depending on whether $u$
and $v$ are adjacent.
This shows that $S_1 = S(i)$ for some $i\in \{2,5\}$.
Now we argue that  $S_2=S(2,3)\cup S(4,5)$.
If $S(3,4)$ contains a vertex $z$, then $z$ is adjacent to $x$ and $y$ by \ref{s22}
but not adjacent to $t$ by \ref{s23}. 
This implies that  $t,x,z,y$ induces a $C_4$, So,  $S(3,4)=\emptyset$.
If $S(1,2)$ contains a vertex $z$,
then $z$ is adjacent to $y$ by \ref{s22} and so $1,z,y,3,4,5,1$ induces a $C_6$, a contradiction. This
shows that $S(1,2)=\emptyset$. By symmetry, $S(5,1)=\emptyset$.
Therefore, $S_2=S(2,3)\cup S(4,5)$. 
The following properties among subsets of $G$ were proved in \cite{GH17}.
\begin{enumerate}[label=\bfseries (\alph*)]
\item {Each vertex in $S(5,1,2)$ is either complete or anti-complete to $S_2$.} \label{item:s512}
\item {$S(2,3)$ and $S(4,5)$ are cliques.} \label{item:s2clique}
\item {Each vertex in $S(3,4,5)\cup S(4,5,1)$ is either complete or anti-complete to $S(4,5)$.
By symmetry, each vertex in $S(1,2,3)\cup S(2,3,4)$ is either complete or anti-complete to $S(2,3)$.}\label{item:s45}
\item $S(4,5)$ is anti-complete to $S(2,3,4)$. By symmetry, $S(2,3)$ is anti-complete to $S(3,4,5)$.\label{item:s234s45}
\item $S(1,2,3)$ is complete to $S(5,1,2)$. By symmetry, $S(5,1,2)$ is complete to $S(4,5,1)$.\label{item:s123s512}
\item {$S(4,5)$ is complete to $S(4,5,1)$. By symmetry, $S(2,3)$ is complete to $S(1,2,3)$. }\label{item:nbr of y}
\item $S(1,2,3)$ is complete to $S(2,3,4)$. By symmetry, $S(3,4,5)$ is complete to $S(4,5,1)$.\label{item:s123s234}
\item $S_5$ is complete to $S_2$.\label{item:s5s2}
\end{enumerate}
Recall that $S_1=S(i)$ for some $i\in \{2,5\}$. By symmetry, we may assume that $S_1=S(5)$.
Note that $S(5)$ is complete to $S(4,5,1)$ by \autoref {lem:P6C4atom} and anti-complete to $S(1,2,3)$
by \ref{s13}. It follows from \ref{s5},  \ref{s33}, \ref{s23},
\ref{item:s123s512}, \ref{item:nbr of y} and \ref{item:s123s234}  that
$S(i-1,i,i+1)\cup \{i\}$ is a homogeneous clique in $G$ and therefore $S(i-1,i,i+1)=\emptyset$ for $i=2,5$.
Similarly, $S(4,5)$ is a homogeneous clique in $G$ by \ref{s23}, \ref{item:s512}-\ref{item:s234s45}, \ref{item:nbr of y} and \ref{item:s5s2}
and so $S(4,5)=\{x\}$. Let  $T=\{t\in S(5,1,2): t \textrm{ is complete to } S_2\}$.
\begin{enumerate}[label=\bfseries (\arabic*)]
\emitem {$S(5)$ is anti-complete to $S(5,1,2)\setminus T$.}\label{item:S5T}%
Let $u\in S(5)$ and $t'\in S(5,1,2)\setminus T$. If $u$ and $t'$ are adjacent, 
then $u,t',2,3,4,x$ induces either a $P_6$ or a $C_6$, depending on whether $u$
and $x$ are adjacent. \e
\end{enumerate}

\noindent
By \ref{item:S5T} and \ref{item:s234s45}, $(S(5,1,2)\setminus T)\cup \{1\}$ is a homogeneous set in $G$
and so $S(5,1,2)\setminus T=\emptyset$.  In other words, $S(5,1,2)$ is complete to $S_2$.
We now partition $S(5)$ into
$X=\{v\in S(5): v \textrm{ has a neighbour in }S(2,3)\}$ and  $ Y=S(5)\setminus X$.
\begin{enumerate}[label=\bfseries (\arabic*)]
\setcounter{enumi}{1}
\emitem {$X$ is anti-complete to $S(3,4,5)$.}\label{item:XS345}%
Let $v\in X$ and $s\in S(3,4,5)$ be adjacent. By the definition of $X$,
$v$ has a neighbour $y'\in S(2,3)$. By \ref{item:s234s45}, $y'$ is not 
adjacent to $s$ and so $v,y',3,s$ induces a $C_4$. \e

\emitem {$X$ is complete to $S(5,1,2)$.}\label{item:XT}%
Assume, by contradiction, that $v\in X$ and $t'\in T$ are not adjacent. 
By the definition of $X$, $v$ has a neighbour $y'\in S(2,3)$. Since $t'$
is adjacent to $y'$, it follows that $v,5,t',y'$ induces a $C_4$. \e

\emitem {$X$ is anti-complete to $Y$.}\label{item:XY}%
Suppose that $u\in X$ and $v\in Y$ are adjacent. Let $y'\in S(2,3)$ be a neighbour
of $u$. Note that $x$ is adjacent to neither $u$ nor $v$ by \ref{s12}.
But now $x,4,3,y',u,v$ induces a $P_6$. \e

\emitem {$X$ is complete to $S_5$.}\label{item:XS5}%
Suppose that $v\in X$ and $u\in S_5$ are not adjacent. Let $y'\in S(2,3)$ be a neighbour
of $v$. By \ref{item:s5s2}, $y'$ and $u$ are adjacent.  Then $u,5,v,y'$ induces a $C_4$.
\e
\end{enumerate}

\noindent
It follows from \ref{s5}-\ref{s23}, \ref{item:s512}-\ref{item:s234s45}, \ref{item:nbr of y}, \ref{item:s5s2} and
\ref{item:XS345}-\ref{item:XS5} that {\color{blue} $(X,S(2,3))$ is a homogeneous pair of sets in $G$.}
\begin{enumerate}[label=\bfseries (\arabic*)]
\setcounter{enumi}{5}

\emitem {For each connected component $A$ of $Y$, each vertex in $S(5,1,2) \cup S(3,4,5)$ is either complete or anti-complete to $A$.}
\label{item:S512Y}%
Let $A$ be an arbitrary connected component of $Y$. Suppose that $s\in S(5,1,2)\cup S(3,4,5)$ distinguishes an edge $aa'$ in $A$, say 
$s$ is adjacent to $a$ but not adjacent to $a'$.  we may assume by symmetry that  $s\in S(5,1,2)$.
Then $a',a,s,2,3,4$ induces a $P_6$, a contradiction. \e

\emitem {Each connected component of $Y$ has a neighbour in both $S(5,1,2)$ and $S(3,4,5)$.}\label{item:YS512S345}%
Suppose that a connected component $A$ of $Y$ does not have a neighbour in one of $S(5,1,2)$ and $S(3,4,5)$, say $S(5,1,2)$.
Then $S_5\cup S(3,4,5)\cup \{5\}$ is a clique cutset of $G$ by \ref{item:XY}.  \e

\emitem {Each connected component of $Y$ is a clique.}\label{item:Y}%
Let $A$ be an arbitrary connected component of $Y$. By \ref{item:YS512S345}, $A$ has a neighbour $s\in S(5,1,2)$
and $r\in S(3,4,5)$. Note that $s$ and $r$ are not adjacent. Moreover, $\{s,r\}$ is complete to $A$ by \ref {item:S512Y}.
Now \ref{item:Y} follows from the fact that $G$ is $C_4$-free. \e

\emitem {$Y$ is complete to $S_5$.}\label{item:YS5}%
Suppose, by contradiction, that $v\in Y$ and $u\in S_5$ are not adjacent. By \ref{item:YS512S345},
$v$ has a neighbour $s\in S(5,1,2)$ and $r\in S(3,4,5)$. Then $v,s,u,r$ induces a $C_4$. \e
\end{enumerate}

\noindent
It follows from \ref{s5}, \ref{item:s5s2}, \ref{item:XS5} and  \ref{item:YS5} that 
each vertex in $S_5$ is a universal vertex in $G$ and so $S_5=\emptyset$.
Let $S'(3,4,5)=\{s\in S(3,4,5): s \textrm{ has a neighbour in }Y\}$ and $S''(3,4,5)=S(3,4,5)\setminus S'(3,4,5)$. 
Note that $S''(3,4,5)$ is anti-complete to $Y$. We now show further properties of $Y$ and $S'(3,4,5)$.

\begin{enumerate}[label=\bfseries (\arabic*)]
\setcounter{enumi}{9}

\emitem {$S'(3,4,5)$ is complete to $S(2,3,4)$.}\label{item:S'345S234}%
Suppose, by contradiction, that $r'\in S'(3,4,5)$ is not adjacent to $s\in S(2,3,4)$.
By the definition of  $S'(3,4,5)$, $r$ has a neighbour $v\in Y$. Then $v,r,4,s,2,1$
induces a $P_6$, a contradiction. \e

\emitem {Each vertex in $S(5,1,2)$ is either complete or anti-complete to $Y$.}\label{item:S512Y2}%
Let $t'\in S(5,1,2)$ be an arbitrary vertex. Suppose that $t'$ has a neighbour $u\in Y$.
Let $A$ be the connected component of $Y$ containing $u$. Then $t'$ is complete to $A$ by
\ref{item:S512Y}. It remains to show that $t'$ is adjacent to each vertex $u'\in  Y\setminus A$.
By \ref{item:YS512S345}, $u$ has a neighbour $s\in S(3,4,5)$. 
Note that $C'=u,t,'y,3,s$ induces a $C_5$. Moreover, $x$ and $s$ are not adjacent
for otherwise $x,s,u,t'$ induces a $C_4$. This implies that $x$ is adjacent only to $t'$
on $C'$. On the other hand, $u'$ is not adjacent to any of $u$, $3$ and $y$. This implies
that $u'$ is adjacent to either $s$ or $t'$ by \autoref{lem:P6C4atom}. 
If $u'$ is not adjacent $t'$, then $u'$ is adjacent to $s$. This implies that $u',s,3,y,t',x$
induces a $P_6$ or $C_6$, depending on whether $u'$ and $x$ are adjacent. 
Therefore, $u'$ is adjacent to $t'$. Since $u'$ is an arbitrary vertex in $Y\setminus A$,
this proves  \ref{item:S512Y2}. \e

\emitem {$S'(3,4,5)$ is anti-complete to $x$.}\label{item:S'345x}%
Suppose not. Let $s\in S'(3,4,5)$ be adjacent to $x$. By definition, $s$ has a neighbour
$y'\in Y$.  Note that $x$ and $y'$ are not adjacent by \ref{s12}. By \ref{item:S512Y}
and	 \ref{item:YS512S345}, $y$ has a neighbour $t\in T=S(5,1,2)$. So, $t$ is adjacent to $x$.
But now $s,y',t,x$ induces a $C_4$, a contradiction. \e
\end{enumerate}

\noindent
It follows from \ref{s5}-\ref{s23}, \ref{item:s234s45}, \ref{item:XS345}, \ref{item:XY}, \ref {item:S'345S234},  \ref{item:S512Y2}
and \ref{item:S'345x}
that {\color{blue} $(Y,S'(3,4,5))$ is a homogeneous pair of sets in $G$.} 

Let $S'(5,1,2)=\{s\in S(5,1,2): s \textrm{ is complete to }Y\}$.
Then $S(5,1,2)\setminus S'(5,1,2)$ is anti-complete to $Y$ by \ref{item:S512Y2}.
It follows from \ref{item:XT} that both $S'(5,1,2)$ and $S(5,1,2)\setminus S'(5,1,2)$  are homogeneous cliques in $G$. 
So, $|S(5,1,2)|\le 2$.  
We now show that $|S''(3,4,5)|\le 1$ and $|S(2,3,4)|\le 1$.
First, we note that if $r\in S(3,4,5)\cap N(x)$, then $r$ is complete to $S(2,3,4)$
for otherwise any non-neighbour $s\in S(2,3,4)$ of $r$ would start an induced $P_6=s,3,r,x,t,1$.
By symmetry, each vertex in $S(2,3,4)\cap N(y)$ is complete to $S(3,4,5)$. 
By \ref{item:s45}, each vertex in $S(2,3,4)\cap N(y)$ is also complete to $S(2,3)$.
Thus, $(S''(3,4,5)\cap N(x))\cup \{4\}$ and $(S(2,3,4)\cap N(y))\cup \{3\}$ are homogeneous cliques in $G$.
So,  $S''(3,4,5)\cap N(x)=S(2,3,4)\cap N(y)=\emptyset$. Namely, $x$ and $y$ are anti-complete
to $S''(3,4,5)$ and $S(2,3,4)$, respectively. By \ref{item:s45}, it follows that $S(2,3)$ is anti-complete to $S(2,3,4)$.
Secondly, $S''(3,4,5)$ and $S(2,3,4)$ are anti-complete to each other. 
If $r\in S''(3,4,5)$ and $s\in S(2,3,4)$ are adjacent, then $x,5,r,s,2,y$ induces a $P_6$ by \ref{item:s234s45}
and the fact that $x$ and $y$ are not adjacent to $r$ and $s$, respectively. 
These two properties imply that each of $S''(3,4,5)$ and $S(2,3,4)$ is a homogeneous clique in $G$.
Hence, $|S''(3,4,5)|\le 1$ and $|S(2,3,4)|\le 1$. 
{\color{blue} Now $G$ is partitioned into a subset $V(C)\cup S(5,1,2)\cup S''(3,4,5)\cup S(2,3,4)\cup \{x\}$ 
of size at most~$10$ and two homogeneous pairs of sets $(X,S(2,3))$ and $(Y,S'(3,4,5))$.}

We now show that each of  $G[X\cup S(2,3)]$ and $G[Y\cup S'(3,4,5)]$ has a nice 4-expression.
First, note that no vertex in $S(1,2)$ can have two non-adjacent neighbours in $X$
since $G$ is $C_4$-free. If there is an induced $P_4=y',x_1,x_2,x_3$ such that $y'\in S(2,3)$ and $x_i\in X$,
then $x_3,x_2,x_1,y',3,4$ induces a $P_6$ in $G$. Now if $P=x_1,x_2,x_3,x_4$ is an induced $P_4$
in $G[X]$, any neighbour $y_1$ of $x_1$ is not adjacent to $x_3$ and $x_4$.
But then $P\cup \{y_1\}$ contains such a labelled $P_4$ in $G[X\cup S(2,3)]$. 
Therefore, $G[X\cup S(2,3)]$ with the partition $(X,S(2,3))$ satisfies all the conditions
of \autoref{lem:recursion} and so has a nice 4-expression.
Finally, note that each vertex in $S(3,4,5)$ can have neighbours in at most one connected component of $Y$
due to \ref{item:YS512S345}, \ref{item:S512Y2} and the fact that $G$ is $C_4$-free.
It then follows from \ref {item:S512Y}-\ref{item:Y} that 
$G[Y\cup S'(3,4,5)]$ with the partition $(Y,S'(3,4,5))$ satisfies all the condition in 
\autoref{lem:recursion} (where $A=S'(3,4,5)$ and $B=Y$) and so has a nice 4-expression.
By \autoref{lem:boundcw}, it follows that $\cw(G)\le 10+2\times 2=14$.
This completes our proof. 
\end{proof}

\begin{lemma}\label{lem:C5}
If a $(C_4,C_6,F_1,F_2,P_6)$-free strong atom $G$ contains an induced $C_5$, then $G$ has clique-width at most $18$.
\end{lemma}

\begin{proof}

Let $C=1,2,3,4,5,1$ be an induced $C_5$ of $G$. 
We partition $V(G)\setminus C$ with respect to $C$.  
We choose $C$ such that 
\begin{itemize}
\item $|S_5|$ is maximized, and\\[-22pt]
\item $|S_3|$ is minimized subject to the above.
\end{itemize}

We first prove the following claim which makes use of the choice of $C$.

\begin{claim}\label{clm:s3}
For each $1\le i\le 5$,  $S(i-1,i,i+1)$ is complete to $S(i,i+1,i+2)$.
\end{claim}

\noindent {\bf Proof of \autoref{clm:s3}.}  
By symmetry, it suffices to prove the claim for $i=1$.
Suppose by contradiction that $S(5,1,2)$ is not complete to $S(1,2,3)$.
Then there exist vertices $s\in S(5,1,2)$ and $t\in S(1,2,3)$ that are not adjacent. 
Consider the induced 5-cycle $C'=C\setminus \{1\}\cup \{s\}$.
Note that $t$ is not a $3$-vertex with respect to $C'$. By the choice of $C$,
there must exist a vertex $r\in V(G)$ that is a $3$-vertex for $C'$ 
but not for $C$. By the definition of $3$-vertices, it follows that $r\in S(2,3)\cup S(4,5)$ and $r$ is adjacent to $s$.
Similarly, by considering the induced 5-cycle $C''=C\setminus \{2\}\cup \{t\}$ we conclude
that there exist a vertex $q\in S(1,5)\cup S(3,4)$ that is adjacent to $t$.
Note that $r$ is not adjacent to $t$ for otherwise $r,t,1,s$ induces a $C_4$.
By symmetry, $q$ is not adjacent to $s$. If $r\in S(2,3)$, then
$4,q,t,1,s,r$ induces a $P_6$ if $q\in S(3,4)$ and $3,r,s,5,q,t$ induces a $C_6$ if $q\in S(5,1)$.
This shows that $r\in S(4,5)$. By symmetry, $q\in S(3,4)$. But now 
$5,1,2,3,r,q$ induces a $C_6$ since $r$ and $q$ are adjacent by \ref{s22}.
This contradicts the assumption that $G$ is $C_6$-free. \e

\begin{claim}\label{cla:S1S34}
If $S(i)$ contains a vertex that is anti-complete to $S(i-2,i+2)$, then $G$ contains an induced $F_2$.
\end{claim}

\noindent {\bf Proof of \autoref{cla:S1S34}.}  By symmetry, we may assume that $i=1$.
First note that $S(1)$ is anti-complete to
$S_1\setminus S(1)$ by \ref{s11} and the $C_6$-freeness of $G$. 
So,  the neighbours of the vertices in $S(1)$ are in $\{1\}\cup S(3,4)\cup S(4,5,1)\cup S(5,1,2)\cup S(1,2,3)\cup S_5$ 
by \ref{s12} and \ref{s13}.
Let $X\subseteq S(1)$ be the set of vertices that have a neighbour in $S(3,4)$, and $S'(1)=S(1)\setminus X$.
Note that $S'(1)\neq \emptyset$ due to our assumption.
Let $a\in S'(1)$ and $A$ be the connected component of $S'(1)$ containing $a$.
If $A$ has a neighbour in both $S(1,2,3)$ and $S(4,5,1)$, it follows from \ref{item:p10} that $G$ contains an induced $F_2$.
Therefore, we may assume that $A$ is anti-complete to $S(4,5,1)$. 
This implies that $N(A)\subseteq \{1\}\cup S(5,1,2)\cup S(1,2,3)\cup S_5 \cup X$.
Note that $S(5,1,2)$ is complete to $S(1,2,3)$ by \autoref{clm:s3}.
So, $\{1\}\cup S(1,2,3)\cup S(5,1,2)\cup S_5$ is a clique by \ref{s5}.
If $x\in X$ distinguishes an edge $aa'$ of $A$, say $x$ is adjacent to $a$
but not to $a'$, then $a',a,s,y,4,5$ induces a $P_6$, where $y\in S(3,4)$ is a neighbour of $x$.
So, each vertex in $X$ is either complete or anti-complete to $A$. Recall that each vertex in $S(1,2,3)$
is either complete or anti-complete to $A$ by \ref{item:p10}.
Let $X'$ and $S'(1,2,3)$ be the set of vertices in $X$ and $S(1,2,3)$ that are complete to $A$, respectively.
If $x'\in X'$ and $s'\in S'_(1,2,3)$ are not adjacent, let $y\in S(3,4)$ be a neighbour of $x'$.
Then either $s',a,x',y$ induces a $C_4$ or $s',a,x',y,4,5$ induces a $P_6$ depending on whether $s'$ and $y$ are adjacent. 
So, $X'$ and $S'(1,2,3)$ are complete.
Suppose that $X'$ contains two non-adjacent vertices $x_1$ and $x_2$.
Let $y_i\in S(3,4)$ be a neighbour of $x_i$ for $i=1,2$.
If $y_1=y_2$, then $y_1,x_1,1,x_2$ induces a $C_4$. So, $y_1\neq y_2$.
This means that $x_1$ is not adjacent to $y_2$ and $x_2$ is not adjacent to $y_1$.
By \ref{s12}, $y_1$ is adjacent to $y_2$ and so $y_1,x_1,a, x_2,y_2$ induces a $C_5$
not dominating $2$, which contradicts \autoref{lem:P6C4atom}.  This shows that $X'$
is a clique. We have proved that $N(A)\setminus S_5$ is a clique.
Since $G$ contains no clique cutset, $A$ must have a pair of non-adjacent neighbours
$u\in S_5$ and $x\in X'$. By the definition of $X'$, $x$ has a neighbour $y\in S(3,4)$.
Note that $u$ is not adjacent to $y$ by \ref{item:s5nodis}. We may choose $a\in A$
to be a neighbour of $u$. But now $\{1,2,3,4,a,x,y,u\}$
induces an $F_2$ (where the underlying 5-cycle is $1,2,3,y,x$). \e

\begin{claim}\label{cla:S1S451S123}
The set $S(i)$ is anti-complete to $S(i-2,i-1,i)\cup S(i,i+1,i+2)$.
\end{claim}

\noindent {\bf Proof of \autoref{cla:S1S451S123}.} By symmetry, we assume that $i=1$.
Suppose, by contradiction, that $x\in S(1)$ is adjacent to some vertex $s\in S(1,2,3)\cup S(4,5,1)$.
We may assume by symmetry that $s\in S(4,5,1)$. By \autoref{cla:S1S34}, $x$ has a neighbour $y\in S(3,4)$. 
Note that $C'=1,x,y,4,5$ induces a 5-cycle and $s$ is complete to $C'$. By the choice of $C$,
there exists $u\in S_5$ such that $u$ is not complete to $C'$. By \ref{item:s5nodis}, 
$u$ is anti-complete to $\{x,y\}$. However, this contradicts \ref{item:s25}. \e

\begin{claim}\label{clm:s2}
If $S(i-2,i+2)\neq \emptyset$ and is anti-complete to $S_1$, then
each vertex in $S(i-2,i+2)$ has a neighbour in $S(i-2,i-1,i)\cup S(i,i+1,i+2)$,
and $S(i-2,i+2)$ is a clique and it is anti-complete to exactly one of  $S(i-2,i-1,i)$ and $S(i,i+1,i+2)$.
\end{claim}

\noindent {\bf Proof of \autoref{clm:s2}.} By symmetry, assume that $S(3,4)\neq \emptyset$.
It follows from \ref{s22},\ref{n22} and the $C_6$-freeness of $G$ that
$S_2=S(3,4)\cup S(i,i+1)$ for some $i\in \{5,1\}$, and $S(3,4)$ is anti-complete  $S_2\setminus S(3,4)$.
We first show that each vertex in $S(3,4)$ has a neighbour in $S(4,5,1)\cup S(1,2,3)$.
Let $Y\subseteq S(3,4)$ be the set of vertices that have a neighbour in $S(4,5,1)\cup S(1,2,3)$,
and $S'(3,4)=S(3,4)\setminus  Y$. Suppose that the claim does not hold, namely $S'(3,4)\neq \emptyset$.
We shall show that $G$ contains a clique cutset and this is a contradiction.
Let $b\in S'(3,4)$ and $B$ be the connected component of $S'(3,4)$ containing $b$. 
First, if $y\in Y$ distinguishes an edge $bb'$ in $B$, say $y$ is adjacent to $b$ but not to $b'$,
then $b',b,y,s,1,2$ or $b',b,y,s,1,5$ induces a $P_6$ where $s\in S(1,2,3)\cup S(4,5,1)$
is a neighbour of $y$. So, any vertex $ Y$ is either complete or anti-complete to $B$.
Let $Y'$, $S'(2,3,4)$ and $S'(3,4,5)$ be the subsets of $Y$, $S(2,3,4)$
and $S(3,4,5)$, respectively that are complete to $B$.
Note that $N(B)=\{3,4\}\cup S'(2,3,4)\cup S'(3,4,5)\cup S_5\cup  Y'$ by \ref{s23}.
If $p\in S'(2,3,4)$ is not adjacent to $q\in S'(3,4,5)$, then $p,b,q,5,1,2,p$ induces a $C_6$,
a contradiction.  This shows that $\{3,4\}\cup S'(2,3,4)\cup S'(3,4,5)\cup S_5$ is a clique by \ref{s5}.

Now let $y_1,y_2\in  Y'$ be two arbitrary and distinct vertices. Suppose that
$y_1$ and $y_2$ are not adjacent. Let $x_i\in S(4,5,1)\cup S(1,2,3)$ be a neighbour of $y_i$
for $i=1,2$. If $x_1=x_2$, then $b,y_1,x_1,y_2$ induces a $C_4$.
So, $x_1\neq x_2$ and this implies that $x_1$ (resp. $x_2$) is not adjacent to $y_2$
(resp. $y_1$).
If $x_1$ and $x_2$ are in $S(4,5,1)$, then  $b,y_1,x_1,x_2,y_2$ induces a $C_5$
not dominating $2$, which contradicts \autoref{lem:P6C4atom}; 
if $x_1\in S(4,5,1)$ and $x_2\in S(1,2,3)$,
then $b,y_1,x_1,1,x_2,y_2,b$ induces a $C_6$, which contradicts our assumption.
So, $Y'$ is a clique. Moreover, $Y'$ is complete to $S'(2,3,4)\cup S'(3,4,5)$ by \ref{item:p14}
and to $S_5$ by \ref{item:s25}. Therefore, $N(B)$ is a clique. 
This proves the first statement of the claim.

Let  $X_2=\{y\in S(3,4): y \text{ has a neighbour in } S(1,2,3)\}$ and
$X_5=\{y\in S(3,4): y \text{ has a neighbour in } S(4,5,1)\}$.
The by \ref{item:p13}, $X_2$ and $X_5$ form a partition of $S(3,4)$, 
and $X_2$ is anti-complete to $S(4,5,1)$ and $X_5$ is anti-complete to $S(1,2,3)$.
If $y\in X_2$ and $z\in X_5$ are adjacent, let $t_2\in S(1,2,3)$ and $t_5\in S(4,5,1)$ be neighbours of $y$ and $z$, respectively.
Then $5,t_5,z,y,t_2,2$ induces a $P_6$, a contradiction. This shows that $X_2$ and $X_5$ are anti-complete to each other.
Suppose now that $x\in X_5$ and $y\in X_2$.  Then $x$ has a neighbour $t\in S(4,5,1)$
and $y$ has a neighbour $s\in S(1,2,3)$. Note also that $x$ and $y$ are not adjacent.
Now $(C\setminus \{5\})\cup \{x,y,s,t\}$ induces a $F_2$ (whose underlying 5-cycle is $y,s,1,t,4$). 
Therefore, we may assume by symmetry that  $X_5=\emptyset$.
In other words, every vertex in $S(3,4)$ has a neighbour in $S(1,2,3)$. 
It remains to show that $S(3,4)$ is a clique.
Suppose that $xs$ and $x's'$ induce a $2P_2$ in the bipartite graph between $S(1,2,3)$ and $S(3,4)$
where $x,x'\in S(3,4)$ and $s,s'\in S(1,2,3)$. Note that $s$ and $s'$ are adjacent since $S(1,2,3)$ is a clique.
Thus, $x$ and $x'$ are not adjacent for otherwise $x,s,s',x'$ induces a $C_4$. Now
$C'=s,x,4,x',s'$ induces a 5-cycle. By \ref{item:s25}, $\{x,x'\}$ is complete to $S_5$ and 
this implies that $S_5$ is complete to $C'$. Moreover, $3$ is complete to $C'$ and this contradicts
the choice of $C$.
Now we can order the vertices in $S(3,4)$ as $x_1,\ldots, x_r$
such that 
$$N_{S(1,2,3)}(x_1)\subseteq N_{S(1,2,3)}(x_2)\ldots \subseteq N_{S(1,2,3)}(x_r).$$
Recall that $x_1$ has a neighbour $t\in S(1,2,3)$ and thus $t$ is complete to $S(3,4)$.
This implies that $S(3,4)$ is a clique since $G$ is $C_4$-free. \e

\medskip
\noindent
We now consider two cases.

\noindent {\bf Case 1.} $S_1=\emptyset$.
By \autoref{clm:s2} and \ref{item:s25}, $S_5$ is complete to $S_2$
and so each vertex in $S_5$ is a universal vertex in $G$ by \ref{s5}.
Therefore, $S_5=\emptyset$.
Recall that  $S(i-1,i,i+1)$ and $S(i,i+1,i+2)$ are complete
for each $1\le i\le 5$ by \autoref{clm:s3}. 
Suppose first that $S_2=\emptyset$. 
Then $S(i-1,i,i+1)\cup \{i\}$ is a homogeneous
clique in $G$ by \ref{s5} and \ref{s33}. So, $S(i-1,i,i+1)=\emptyset$.
Now $G$ is the 5-cycle and has clique-width at most $4$.

Suppose now that $S_2\neq \emptyset$, say $S(3,4)\neq \emptyset$.
By \autoref{clm:s2}, we may assume that $S(3,4)$ is anti-complete to $S(4,5,1)$,
each vertex in $S(3,4)$ has a neighbour in $S(1,2,3)$,
and $(S(3,4),S(1,2,3))$ is a co-bipartite chain graph.
Let $x'\in S(3,4)$ and $s'\in S(1,2,3)$ be a neighbour of $x'$.
If $S(1,5)$ contains a vertex $y$, then either $2,s',x',4,5,y$ induces a $P_6$
or $C\cup \{x',y,s'\}$ induces a subgraph isomorphic to $F_2$, depending on whether $s'$ and $y$
are adjacent. This shows that $S(1,5)=\emptyset$. In addition, $S(2,3)=S(4,5)=\emptyset$ by
\ref{s22} and the $C_6$-freeness of $G$. So, $S_2=S(3,4)\cup S(1,2)$.
Suppose that $x\in S(3,4)$ and $t\in S(3,4,5)$ are not adjacent. 
Let $s\in S(1,2,3)$ be a neighbour of $x$. Then $C\cup \{x,s,t\}$
induces a subgraph isomorphic to $F_2$ (where the underlying 5-cycle is $x,s,1,5,4$). 
This shows that $S(3,4)$ is  complete to $S(3,4,5)$.

Note that $S(1,2)$ may not be empty.
If $S(1,2)\neq \emptyset$, then it is anti-complete to $S(4,5,1)$ by \ref{item:p15}
and the fact that $S(4,5,1)$ is anti-complete to $S(3,4)$.
It then follows from \autoref{clm:s2} that  each vertex in $S(1,2)$ has a neighbour in $S(2,3,4)$,
and $(S(1,2),S(2,3,4))$ is a co-bipartite chain graph. In addition, the above argument for $S(3,4)$
works symmetrically for $S(1,2)$. In particular, $S(1,2)$ is complete to $S(5,1,2)$.
This implies that $S(i-1,i,i+1)\cup \{i\}$ for each $i=1,4,5$ is a homogeneous clique in $G$ by \ref{s33} and \ref{s23}.
Therefore, $S(i-1,i,i+1)=\emptyset$ for $i=1,4,5$. 
Let $S'(2,3,4)$ be the set of vertices in $S(2,3,4)$ that are complete to $S(3,4)$
and $S''(2,3,4)=S(2,3,4)\setminus S(2,3,4)$. Similarly,  if $S(1,2)\neq \emptyset$,
let $S'(1,2,3)$ be the set of vertices in $S(1,2,3)$ that are complete to $S(1,2)$
and $S''(1,2,3)=S(1,2,3)\setminus S'(1,2,3)$. In case that $S(1,2)=\emptyset$,
we define $S'(1,2,3)=S(1,2,3)$ and $S''(1,2,3)=\emptyset$.
By \ref{item:p14} and the way we define $S'(1,2,3)$ and $S''(1,2,3)$,
it follows that $S''(2,3,4)$ and $S''(1,2,3)$ are anti-complete to $S(3,4)$ and $S(1,2)$, respectively.
Let $x\in S(3,4)$. If $s\in S''(2,3,4)$ is adjacent to some $y\in S(1,2)$, then
$x,3,s,y,1,5$ induces a $P_6$.
So, $S''(2,3,4)$ is anti-complete to $S(1,2)$. By symmetry,  $S''(1,2,3)$ is anti-complete to $S(3,4)$
(note that this is vacuously true if $S(1,2)=\emptyset$ since we define $S''(1,2,3)=\emptyset$).
This implies that $S''(2,3,4)$ and $S''(1,2,3)$ are homogeneous cliques in $G$ and so
both have size at most $1$. Furthermore, $(S'(2,3,4), S(1,2))$ and $(S'(1,2,3),S(3,4))$ are
homogeneous pairs of cliques in $G$. 
{\color{blue} Now $V(G)$ is partitioned into a subset $C\cup S''(2,3,4)\cup S''(1,2,3)$
of size at most $7$ and two homogeneous pairs of cliques} each of which has a nice 4-expression 
by \autoref{lem:co-bipartite chain graph}. Therefore, $\cw(G)\le 7+2\times 2=11$ by \autoref{lem:boundcw}.
This completes the proof of Case 1.

\medskip
\noindent {\bf Case 2.} $S_1\neq \emptyset$. 
By symmetry, we assume that $S(4)\neq \emptyset$.
It follows from \autoref{cla:S1S34} and \autoref{cla:S1S451S123} that each vertex in $S(4)$ has a neighbour in $S(1,2)$
and $S(4)$ is anti-complete to $S(2,3,4)\cup S(4,5,1)$. 
In particular, $S(1,2)$ contains a vertex that has a neighbour in $S(4)$.
By \ref{s12}, this vertex is universal in $S(1,2)$ and so $S(1,2)$ is connected.
Note that $S_1=S(4)$  by \ref{s11}, \ref{n12} and the $C_6$-freeness of $G$,
and $S_2=S(1,2)\cup S(i,i+1)$ for some $i=3,4$ by \ref{s22}, \ref{n22} and the $C_6$-freeness of $G$.
By symmetry, we can assume that $S_2=S(1,2)\cup S(3,4)$.
Moreover,  $S(i-1,i,i+1)$ is complete to $S(i,i+1,i+2)$ for each $1\le i\le 5$
by \autoref{clm:s3}.
If some vertex $x\in S(3,4)$ has a neighbour $s\in S(4,5,1)$, then $s$ is complete to $S(1,2)$ by \ref{item:p15}.
Let $v\in S(4)$ and $y\in S(1,2)$ be a neighbour of $v$. Then $4,v,y,s$ induces a $C_4$.
This shows that $S(3,4)$ is anti-complete to $S(4,5,1)$. 
Recall that $S(3,4)$ is anti-complete to $S(4)$ by \ref{s12}.
Thus, it follows from \autoref{clm:s2} that each vertex in $S(3,4)$ has a neighbour in $S(1,2,3)$
and $(S(3,4), S(1,2,3))$ is a co-bipartite chain graph if $S(3,4)\neq \emptyset$.

Let $S'(3,4,5)$ be the set of vertices in $S(3,4,5)$
that are complete to $S(3,4)$ and $S''(3,4,5)=S(3,4,5)\setminus S'(3,4,5)$.
Then $S''(3,4,5)$ is anti-complete to $S(3,4)$ by the fact that $S(3,4)$ is a clique and \ref{item:p14}.
Then $S'(3,4,5)\cup \{4\}$ and $S''(3,4,5)$ are homogeneous cliques in $G$
by \ref{s5}, \ref{s33},  \ref{s23} and \autoref{lem:P6C4atom}.
So, $S'(3,4,5)=\emptyset$ and $|S''(3,4,5)|\le 1$.
This implies that $|S(3,4,5)|\le 1$.

\begin{enumerate}[label=\bfseries (\arabic*)]

\emitem {$S(1,2,3)$ is complete to $S(1,2)$. By symmetry $S(5,1,2)$ is complete to $S(1,2)$. }\label{item:S123S12}%
Let $s\in S(1,2,3)$. Note that there is an edge $yz$ between $S(4)$ and $S(1,2)$
where $y\in S(1,2)$ and $z\in S(4)$. We note that $s$ and $y$ are adjacent for otherwise
$1,y,z,4,3,s$ induces a $C_6$. On the other hand, since $S(1,2)$ is connected,
it follows from \autoref{cla:S1S451S123} $s$ is complete to $S(1,2)$. \e
\end{enumerate}

By \ref{item:S123S12}, $S(5,1,2)\cup \{1\}$ is a homogeneous clique in $G$ and so $S(5,1,2)=\emptyset$.
Recall that $S(3,4)$ is a clique if it is not empty and so no vertex in $S(2,3,4)\cup S(3,4,5)$ can distinguish
two vertices in $S(3,4)$ by \ref{item:p14}. It then follows from \ref{s5}, \ref{s22}, \ref{s12}, \ref{s13}, \ref{s23}
and \ref{item:S123S12} that {\color{blue} $(S(1,2,3),S(3,4))$ is a homogeneous pair of cliques in $G$}
(if $S(3,4)=\emptyset$ this means that $S(1,2,3)$ is a homogeneous clique). 

\begin{enumerate}[label=\bfseries (\arabic*)]
\setcounter{enumi}{1}
\emitem {Each vertex in $S(1,2)$ is anti-complete to either $S(4)$ or $S(2,3,4)\cup S(4,5,1)$.}\label{item:S12antic}%
Suppose that $d\in S(1,2)$ has a neighbour $v\in S(4)$ and $s\in S(2,3,4)\cup S(4,5,1)$.
By \autoref{cla:S1S451S123}, $v$ and $s$ are not adjacent. But now $4,v,d,s$ induces a $C_4$, a contradiction. \e
\end{enumerate}

We let $X=\{v\in S(1,2): v \textrm{ has a neighbour in } S(4)\}$. We let
$Y=\{v\in S(1,2): v \textrm{ has a neighbour in }S(2,3,4)\cup S(4,5,1)\}$,
and $Z=S(1,2)\setminus (X\cup Y)$. By \ref{item:S12antic}, $X$, $Y$ and $Z$
form a partition of $S(1,2)$. Note that $X$ is anti-complete to $S(2,3,4)\cup S(4,5,1)$,
$Y$ is anti-complete to $S(4)$ and $Z$ is anti-complete to $S(2,3,4)\cup S(4,5,1)\cup S(4)$.
By \ref{s12}, each vertex in $X$ is universal in $S(1,2)$.
By \ref{item:s25}, $Y$ is complete to $S_5$. 
Suppose that $y\in Y$ is not adjacent to $z\in Z$. 
Let $s\in S(2,3,4)\cup S(4,5,1)$ be a neighbour of $y$.
By symmetry, we assume that $s\in S(2,3,4)$. 
Recall that $S(4)$ contains a vertex, say $v$. Then $z,1,y,s,4,v$ induces a $P_6$.
This shows that $Y$ is complete to $Z$.
Suppose that $Y$ contains a vertex $y'$ that has a neighbour $s\in S(4,5,1)$ and 
a vertex $y''$ that has a neighbour $t\in S(2,3,4)$. Then $s$ (resp. $t$) is not adjacent to
$y''$ (resp. $y'$) by \ref{item:p13}. Then $C\cup \{s,t,y',y''\}$ contains an induced $P_6$ or an induced $F_2$,
depending on whether $y'$ and $y''$ are adjacent. 
Thus, either each vertex in $Y$ is anti-complete to $S(4,5,1)$ and has a neighbour in $S(2,3,4)$
or each vertex in $Y$ is anti-complete to $S(2,3,4)$ and has a neighbour in $S(4,5,1)$.
Then $Y$ is a clique by a similar argument in \autoref{clm:s2} asserting that there is a vertex
in $S(2,3,4)$ or $S(4,5,1)$ that is complete to $Y$. 

Recall that $S(3,4)$ is anti-complete to $S(4,5,1)$.
Let $R\subseteq S(2,3,4)$ be the set of vertices that are complete to $S(3,4)$ and $T=S(2,3,4)\setminus R$.
By \ref{item:p10}, $T$ is anti-complete to $S(3,4)$.
If $Y$ is anti-complete to $S(2,3,4)$, then $R\cup \{3\}$ and $T$ are homogeneous cliques in $G$. Hence, $|S(2,3,4)|\le 1$.
Moreover,   {$(Y,S(4,5,1))$ is a homogeneous pair of cliques in $G$} by \ref {item:S123S12}.
If $Y$ is anti-complete to $S(4,5,1)$, then $S(4,5,1)\cup \{5\}$ is a homogeneous clique in $G$ and so $S(4,5,1)=\emptyset$. 
If $S(3,4)=\emptyset$, then $(Y,S(2,3,4))$ is a homogeneous pair of cliques in $G$. Otherwise let $d\in S(3,4)$.
We claim that $T$ is anti-complete to $Y$.
Suppose by contradiction that $y\in Y$ and $t\in T$ are adjacent.  Then $d$ and $t$ are not adjacent by the definition
of $T$. Note that $S(4)$ contain a vertex $z$ and $z$ has a neighbour $x\in X\subseteq S(1,2)$.
Note that $x$ and $y$ are adjacent and $y$ is not adjacent to $z$ and $t$ is not adjacent to $x$ by 
\ref{item:S12antic}. But now $d,3,t,y,x,z$ induces a $P_6$. So, $T$ is anti-complete to $Y$.
This implies that {$(Y,R)$ is a homogeneous pair of cliques in $G$}, and $T$ is a homogeneous clique in $G$ and so $|T|\le 1$.
In either case, 
{\color{blue} $Y\cup S(2,3,4)\cup S(4,5,1)$ is partitioned into a subset of size at most $1$ and a homogeneous pair of cliques in $G$.}

We next deal with $S_5\cup S(4)\cup X\cup Z$.
Let $S''_5=\{u\in S_5: u \textrm{ has a non-neighbour in }X\}$ and $S'_5=S_5\setminus S''_5$. 
Then $S'_5$ is complete to $X$. 

\begin{enumerate}[label=\bfseries (\arabic*)]
\setcounter{enumi}{2}
\emitem {Each vertex in $Z$ has a neighbour in $S''_5$.}\label{item:ZS''5}%
Let $Z'=\{z\in Z: z \textrm{ has a neighbour in }S''_5\}$ and $Z''=Z\setminus Z'$.
Suppose that $Z''$ contains a vertex $a$ and let $A$ be the connected component of $Z''$ containing $a$. 
Thus, $N(A)\subseteq \{1,2\}\cup X\cup Y\cup S(1,2,3)\cup S'_5\cup Z'$.
Recall that $Y$ is complete to $S_5$ and so $\{1,2\}\cup X\cup Y\cup S(1,2,3)\cup S'_5$ is a clique by \ref{item:S123S12}.
Moreover, $Y$ is complete to $Z$ and so $Z'$ is complete to $\{1,2\}\cup X\cup Y\cup S(1,2,3)$. 
We now show that $N(A)$ is a clique.
First, each vertex in $Z'$ is either complete or anti-complete $A$.
Suppose not. Let $z'\in Z'$ distinguish an edge $aa'$ in $A$, say $z'$ is adjacent to $a$ but not to $a'$.
Let $u\in S''_5$ be a neighbour of $z'$ and $x\in X$ be a non-neighbour of $u$. 
Then $x$ has a neighbour $v\in S(4)$. By \ref{item:s5nodis}, $u$ is not adjacent to $v$.
So, $a',a,z',u,4,v$ induces a $P_6$. Second, let $u\in S'_5$ and $z'\in Z'$ be in $N(A)$.
We may choose $a\in A$ to be a neighbour of $u$. By the definition of $Z'$, $z'$ has a neighbour $u''\in S''_5$.
Moreover, $u''$ is not adjacent to $a$ by the definition of $Z''$. Thus, $u$ and $z'$ must
be adjacent for otherwise $a,z',u'',u$ induces a $C_4$. This shows that $Z'$ and $S'_5$ are complete.
Thirdly, let $z_1,z_2\in N(A)\cap Z'$.
Suppose that $z_1$ and $z_2$ are not adjacent.  Let $u_i\in S''_5$ be a neighbour of $z_i$
and $x_i\in X$ be a non-neighbour of $u_i$. By the definition of $X$, $x_i$ has a neighbour 
$v_i\in S(4)$. By \ref{item:s5nodis}, $u_i$ is not adjacent to $v_i$.
Note that  $u_1\neq u_2$ for otherwise $a,z_1,u_1,z_2$ induces a $C_4$. This means
that $u_1$ (resp. $u_2$) is not adjacent to $z_2$ (resp. $z_1$). Now
$z_2,a,z_1,u_1,4,v_1$ induces a $P_6$. This shows that $N(A)$ is a clique
and so a clique cutset of $G$. Since $G$ has no clique cutsets, $Z''=\emptyset$.\e 

\emitem {$Z$ is complete to $S'_5$.}\label{item:ZS'5}%
Let $z\in Z$ be an arbitrary vertex. Then $z$ has a neighbour $u\in S''_5$
which has a non-neighbour $x\in X$. If $u'\in S'_5$
is not adjacent to $z$, then $u,u',x,z$ induces a $C_4$. \e
\end{enumerate}

Since $S'_5$ is complete to $X$ and each vertex in $S(4)$ has a neighbour in $X$,
$S'_5$ is complete to $S(4)$ by \ref{item:s5nodis}. So, \ref{item:ZS'5} implies that
each vertex in $S'_5$ is a universal vertex in $G$ (note that if $S(3,4)\neq \emptyset$ then it is also complete to $S_5$).
Therefore, $S'_5=\emptyset$ and $S_5=S''_5$.
We have shown that each vertex in $S_5$ has a non-neighbour in $X$ and each vertex
in $Z$ has a neighbour in $S_5$. We proceed with partitioning $X$, $S''_5$ and $Z$
into subsets so that we can decompose $G$ into homogeneous pairs of sets.
Let $X_0=\{x\in X: x \textrm{ is anti-complete to $S''_5$}\}$ and $X_1=X\setminus X_0$. 
We then partition $S''_5$ into two subsets  $S'''_5=\{u\in S''_5: u \textrm { is complete to }X_1\}$ and $R=S''_5\setminus S'''_5$.
Let $Z'=N(R)\cap Z$ and $Z''=Z\setminus Z'$.
\begin{enumerate}[label=\bfseries (\arabic*)]
\setcounter{enumi}{4}
\emitem {$N(X_0)\cap S(4)$ is anti-complete to $X_1$.}\label{item:X1}%
Suppose not. Let $v\in S(4)\cap N(X_0)$ have a neighbour $x_1\in X_1$.
Let $x_0\in X_0$ be a neighbour of $v$. By the definition of $X_1$, $x_1$
has a neighbour $u\in S''_5$. Note that $u$ is not adjacent to $x_0$ and so
not adjacent to $v$ by \ref{item:s5nodis}. But now $x_1,v,4,u$ induces a~$C_4$, a contradiction. \e
\end{enumerate}

Recall that $X$ is a clique. Therefore, $S(4)\setminus N(X_0)$ and $S(4)\cap N(X_0)$ are anti-complete
by \ref{item:X1} and the facts that $G$ is $C_4$-free and each vertex in $S(4)$ has a neighbour in $X$.
In addition, $S(4)\cap N(X_0)$ is anti-complete to $S''_5$ by \ref{item:s5nodis}. Therefore, 
{\color{blue} $(X_0,N(X_0)\cap S(4))$ is a homogeneous pair of sets in $G$.}

\begin{enumerate}[label=\bfseries (\arabic*)]
\setcounter{enumi}{5}

\emitem {No vertex in $S''_5$ can have two non-adjacent neighbours in $Z$.}\label{item:ZS5''}%
Let $u\in S''_5$ have two non-adjacent neighbours $z_1$ and $z_2$ in $Z$.
Then $u$ has a non-neighbour $x\in X$. Recall that $x$ is universal in $S(1,2)$
and so complete to $\{z_1,z_2\}$. Now $u,z_1,x,z_2$ induces a $C_4$. \e

\emitem {$Z'$ and $Z''$ are complete.}\label{item:Z'Z''}%
Suppose not.  Let $z'\in Z'$ and $z''\in Z''$ be non-adjacent.
By the definition of $Z'$, $z'$ has a neighbour $r\in R\subseteq S''_5$.
Let $u\in S'''_5$ be a neighbour of $z''$. By the definition of $R$, $r$
is not adjacent to some vertex $x_1\in X_1$. 
Since $X$ is complete to $Z$,
$x_1$ is adjacent to $z'$ and $z''$. Moreover, $x_1$ is adjacent to $u$
by the definition of $X_1$. 
By \ref{item:ZS5''}, $u$ is not adjacent to $z'$.
But now $x_1,z',r,u$ induces a $C_4$.
\e

\emitem {$Z'$ is complete to $S'''_5$.}\label{item:Z'S'''5}%
Suppose not. Let $z'\in Z'$ and $u\in S'''_5$ be non-adjacent. 
By the definition of $Z'$, $z'$ has a neighbour $r\in R$, and 
$r$ is not adjacent to some vertex $x_1\in X_1$. 
Since $X$ is complete to $Z$, $x_1$ is adjacent to $z'$.
Moreover, $u$ is adjacent to $x_1$ and $r$. So, $u,r,z',x_1$
induces a $C_4$, a contradiction.\e
\end{enumerate}

Since each vertex in $S(4)\setminus N(X_0)$ has a neighbour in $X_1$, 
$S'''_5$ is complete to $S(4)\setminus N(X_0)$ by \ref{item:s5nodis}.
It follows from  \ref{item:Z'Z''} and \ref{item:Z'S'''5} that
{\color{blue} $(S'''_5,Z'')$ is a homogeneous pair of sets in $G$.}
Moreover,  {\color{blue} $(R,X_1\cup Z', S(4)\setminus N(X_0))$ is a homogeneous triple in $G$}
by the structural properties we have proved so far.
Therefore,
{\color{blue} $V(G)$ is partitioned into a subset $V_0$ consisting of vertices in $C\cup S(3,4,5)$ and at most one vertex
in $Y\cup S(2,3,4)\cup S(4,5,1)$  (and so of size at most $7$),
four homogeneous pairs of sets $(S(1,2,3),S(3,4))$, one contained in $Y\cup S(2,3,4)\cup S(4,5,1)$,  
$(X_0,N(X_0)\cap S(4))$ and $(S'''_5,Z'')$, and  a homogeneous triple $R\cup (X_1\cup Z')\cup (S(4)\setminus N(X_0))$.}

We now show that 
each of the four homogenous pairs and the homogenous triple has 
a nice 4-expression and 6-expression, respectively.
First of all, $(S(1,2,3),S(3,4))$ and the homogeneous pair of sets contained in $Y\cup S(2,3,4)\cup S(4,5,1)$  
are homogeneous pairs of cliques and so have a nice 4-expression by \autoref{lem:co-bipartite chain graph}.

Recall that each vertex in $Z$ has a neighbour in $S''_5$.
Suppose that $P=u,z_1,z_2,z_3$ is an induced $P_4$ where $u\in S''_5$ and $z_1,z_2,z_3\in Z$.
Then $u$  has a non-neighbour $x\in X$, and $x$ ha a neighbour $v\in S(4)$.
By \ref{item:s5nodis}, $u$ and $v$ are not adjacent. Now $v,4,P$ induces a $P_6$. 
Suppose now that $P=z_1z_2z_3z_4$ be an induced $P_4$ in $Z$.
Then $z_1$ has a neighbour $u\in S''_5$ which has a non-neighbour $x\in X$.
By \ref{item:ZS5''}, $u$ is not adjacent to $z_3$ or $z_4$.
Now $P\cup \{u\}$ contains such a labeled $P_4$.
Therefore, $G[S''_5\cup Z]$ with the partition $(S''_5,Z)$ satisfies all the conditions
in \autoref{lem:recursion} and so has a nice 4-expression. 
Since $G[S'''_5\cup Z'']$ is an induced subgraph of $G[S''_5\cup Z]$, 
it follows that $G[S'''_5\cup Z'']$ has  a nice 4-expression. 

\begin{enumerate}[label=\bfseries (\arabic*)]
\setcounter{enumi}{8}
\emitem {$Z'$ is a clique.}\label{item:Z'clique}%
Suppose that $z_1,z_2\in Z'$ are not adjacent. 
Then $z_i$ has a neighbour $r_i\in R$ for $i=1,2$.
By \ref{item:ZS5''}, $r_1\neq r_2$ and $r_1$ (resp. $r_2$)
is not adjacent to $z_2$ (resp. $z_1$). 
Recall that $r_1$ has a non-neighbour $x\in X_1$. 
Since $x$ is adjacent to $z_1$ and $z_2$, $x$ is not adjacent to $r_2$.
Now by the definition of $X_1$, $x$ has a neighbour $u\in S''_5$.
Then $u$ is adjacent to $z_i$ for otherwise $x,z_i,r_i,u$ induces a $C_4$
for $i=1,2$. This, however, contradicts \ref{item:ZS5''}. \e
\end{enumerate}

\begin{claim}\label{clm:S5XS4}
$G[S_5\cup (X\cup Z')\cup S(4)]$ has a nice $6$-expression.
\end{claim}

\noindent {\bf Proof of \autoref{clm:S5XS4}.}
Let $X'=X\cup Z'$. By \ref{item:Z'clique}, $(S_5,X')$ is a co-bipartite chain graph. 
Thus, we can order the vertices in $S_5$ as $ u_1,\ldots,u_r$ and the vertices in $X'$
as $x_1,\ldots,x_s$ such that $N_{X'}(u_i)=\{x_1,\ldots, x_j\}$ for some $0\le j\le s$
and $N_{X'}(u_1)\subseteq N_{X'}(u_1)\subseteq \ldots \subseteq N_{X'}(u_r)$.
Let $U_0$ be the set of vertices in $S_5$ that are anti-complete to $X'$.
Let $U_1$ be the set of vertices in $S_5\setminus U_0$ that have the smallest neighbourhood in $X'$, 
and $U_2$ be the set of vertices in $S_5\setminus U_0$ that have the second smallest neighbourhood in $X'$, and so on.  
Then $S_5$ is partitioned into $U_0,U_1,\ldots,U_q$ where $q+1$ is the number of different neighbourhoods of vertices in $S_5$.
Let $X_1=N_{X'}(U_1)$ and $X_i=N_{X'}(U_{i})\setminus N_{X'}(U_{i-1})$
for $2\le i\le q$. Let $X_{q+1}$ be the set of vertices in $X'$ that are anti-complete to $S_5$.
Note that $X_1,\ldots X_q,X_{q+1}$ partition $X'$. 
Let $M_i=N(X_i)\cap S(4)$ for each $1\le i\le q+1$.
Since $S(4)$ is anti-complete to $Z$, some $M_i$ may be empty.
Since each vertex in $S(4)$ has a neighbour in $X$, $S(4)=M_1\cup \ldots \cup M_{q+1}$.
We say that $G[ X_j\cup M_j]$ is a \emph{piece} for each $1\le i\le q+1$.
Note that any vertex in $M_j$ has a neighbour in $X_j\cap X$ since $Z$ is anti-complete to $S(4)$.
\begin{enumerate}[label=\bfseries (\roman*)]
\emitem {$U_0$ is anti-complete to $X'\cup S(4)$.}\label{item:U0}%
By the definition of $U_0$, $U_0$ is anti-complete to $X'$. 
Since any vertex in $S(4)$  has a neighbour in $X$,  $U_0$ is anti-complete to $S(4)$ by \ref{item:s5nodis}. \e

\emitem {For $1\le i\le q$, $U_i$ is complete to $X_j\cup M_j$ for $1\le j\le i$.}\label{item:ucomp}%
Let $u\in U_i$ and $1\le j\le i$.  Then $u$ is complete to $X_j$ by the definition of $X_j$.
If $M_j$ is empty, then the proof is complete. So, we assume that $M_j\neq \emptyset$.
For each $v\in M_j$, $v$ has a neighbour in $X_j\cap X$. Since $u$ is complete to $X_j$,
$u$ is adjacent to $v$ by \ref{item:s5nodis}. This completes the proof. \e

\emitem {For each $1\le i\le q$, $U_i$ is anti-complete to $X_j\cup M_j$ for $i<j\le q+1$.}\label{item:uantic}%
Let $u\in U_i$ and let $j$ be such that $i<j\le q+1$. Then $u$ is anti-complete to $X_j$ by the definition of $X_j$.
If $M_j$ is empty, then the proof is complete. So, we assume that $M_j\neq \emptyset$.
For each $v\in M_j$, $v$ has a neighbour in $X_j\cap X$. Since $u$ is anti-complete to $X_j$,
$u$ is not adjacent to $v$ by \ref{item:s5nodis}. \e

\emitem {For $1\le i,j\le q+1$ with $i\neq j$,  $M_i\cap M_j=\emptyset$.}\label{item:disjoint}%
Suppose that $i<j$ and $v\in M_i\cap M_j$. 
Then $v$ has a neighbour $x_i\in X_i$ and a neighbour $x_j\in X_j$. Let $u_i\in U_i$. 
By \ref{item:ucomp}, $u_i$ is adjacent to $v$, and  by \ref{item:uantic} $u_i$ is not adjacent to $v$.
This is a contradiction. \e

\emitem {For $1\le i,j\le q+1$ with $i\neq j$,  $M_i$ and $M_j$ are anti-complete.}\label{item:antic}%
Suppose that $v_i\in M_i$ and $v_j\in M_j$ are adjacent. Then $v_i$ has a neighbour $x_i\in X_i$
and $v_j$ has a neighbour $x_j\in X_j$. By \ref{item:disjoint}, $v_i$ is not adjacent to $x_j$ and
$v_j$ is not adjacent to $x_i$ . Then $v_i,x_i,x_j,v_j$ induces a $C_4$. \e
\end{enumerate}

Since no vertex in $S(1,2)$ can have two non-adjacent neighbours in $S(4)$,
each piece $G[ X_i\cup M_i]$ with the partition $(X_i,M_i)$ satisfies all the conditions in \autoref{lem:recursion}
and so there is a nice $4$-expression $\tau_i$ for it
where all vertices in $X_i$ have label $2$ and  all vertices in $M_i$ have label $4$.
For $0\le i\le q$, let $\epsilon_i$ be a $2$-expression for $U_i$ in which all vertices in $U_i$ have label $5$.
We now construct a  nice $6$-expression for $G[S_5\cup (X\cup Z')\cup S(4)]$.
Let $\sigma_1=\rho_{5\rightarrow 6}(\eta_{5,4}(\eta_{5,2}(\epsilon_1\oplus \tau_1)))$.
For each $2\le i\le q$, let 
$$\sigma_i=\rho_{5\rightarrow 6}(\eta_{5,6}(\eta_{5,4}(\eta_{5,2}
(\epsilon_i\oplus \rho_{1\rightarrow 2}(\eta_{1,2}(\sigma_{i-1}\oplus \rho_{2\rightarrow 1}(\tau_i))))))).$$
Then
$\sigma=\eta_{5,6}(\epsilon_0\oplus (\rho_{1\rightarrow 2}(\eta_{1,2}(\sigma_q\oplus \rho_{2\rightarrow 1}(\tau_{q+1})))))$
is a nice $6$-expression for $G[S_5\cup (X\cup Z')\cup S(4)]$
(the correctness of the construction follows from \ref{item:U0}-\ref{item:antic}).
This completes the proof of the claim. \e

\medskip
Since $(X_0,N(X_0)\cap S(4))$ and $R\cup (X_1\cup Z')\cup (S(4)\setminus N(X_0))$
induce subgraphs of $G[S_5\cup (X\cup Z')\cup S(4)]$, the pair has a nice 4-expression and the triple has a nice 6-expression.
by \autoref{clm:S5XS4}. Finally, $\cw(G)\le |V_0|+2\times 4+3=18$.
This completes our proof of the lemma.
\end{proof}

We are now ready to prove our main theorem.

\begin{proof}[Proof of \autoref{thm:p6c4atom}]
Let $G$ be a $(C_4,P_6)$-free atom.
Let $G'$ be the graph obtained from $G$ by removing twin vertices and universal vertices.
It follows from \autoref{lem:F1}--\autoref{lem:C5} that if $G'$ contains an induced $C_5$ or $C_6$,
then $G'$ has clique-width at most 18. Therefore, we can assume that $G'$ is also $(C_5,C_6)$-free and
therefore $G'$ is chordal. It then follows from a well-known result of Dirac \cite{Di61} that $G'$ is a clique
whose clique-width is $2$.  Finally, $\cw(G)=\cw(G')$ by  \autoref{lem:prime} and this completes the proof.
\end{proof}

\section{The Hardness Result}\label{sec:hard}

In this section, we prove that \cn is NP-complete on the class of $(C_4,3P_3,P_3+P_6,2P_5,P_9)$-free graphs.
A graph is a {\it split graph} if its vertex set can be partitioned into two disjoint sets $C$ and $I$
such that $C$ is a clique and $I$ is an independent set. The pair $(C,I)$ is called a {\it split partition} of $G$. 
A split graph is {\it complete} if it has a {\it complete} split partition, 
that is, a partition $(C,I)$ such that $C$ and $I$ are complete to each other.
It is straightforward to check that a complete split graph is $P_4$-free.
A {\it list assignment} of a graph $G=(V,E)$ is a function $L$ that prescribes, for each $u\in V$, a finite list $L(u)\subseteq \{1, 2, \dots \}$ of colours for $u$.
The {\it size} of a list assignment ${L}$ is the maximum list size $|L(u)|$ over all vertices $u\in V$.
A colouring $c$ {\it respects} ${L}$ if  $c(u)\in L(u)$ for all $u\in V$.
The \lc problem is to decide whether a given graph $G$ has a colouring $c$ that respects a given list assignment~$L$.

For our hardness reduction we need the following result.
\begin{lemma}[\cite{GP14}]\label{l-gp}
\lc is \NP-complete even for complete split graphs with a list assignment of size at most~$3$.
\end{lemma}

We are now ready to prove our theorem. 
In its proof we construct a $C_4$-free graph~$G'$ that is neither $(sP_2+P_8)$-free nor $(sP_2+P_4+P_5)$-free for any $s\geq 0$.
Hence, a different construction is needed for tightening our hardness result (if possible).

\medskip
\noindent
\begin{theorem}
\cn is \NP-complete for $(C_4,3P_3,P_3+P_6,2P_5,P_9)$-free graphs.
\end{theorem}
\begin{proof}
We reduce from \lc on complete split graphs with a list assignment of size at most~$3$.
This problem is \NP-complete due to \autoref{l-gp}.
Let~$G$ be a complete split graph with a list assignment~$L$ of size at most~3.  
From $(G,L)$ we construct an instance $(G',k)$ of \cn as follows. 
First, let $k\le 3|V(G)|$ be the size of the union of all lists $L(u)$.
Let $(C,I)$ be a complete split partition of $V(G)$. We say that the vertices of $C$ and $I$ are of $c$-type and $i$-type, respectively.
Let $G'$ be the graph of size $O(|V(G)| k)$ obtained from $G$ by adding the following sets of vertices and edges (see Figure~\ref{fig:hard} for an illustration):

\begin{itemize}
\item Take a clique $X$ on $k$ vertices $x_1,\ldots,x_k$. We say that these vertices are of $x$-type.
\item For each $u\in V(G)$, introduce a clique $Y_u$ of size $k-|L(u)|$ such that every vertex of $Y_u$ is adjacent to
$u$ and also to every $x_i$ with $i \in L(u)$ (so, each vertex in every $Y_u$ is adjacent to exactly one vertex of 
$V(G)$, namely vertex $u$). We say that the vertices in every $Y_u$ are of $y$-type.
\end{itemize}

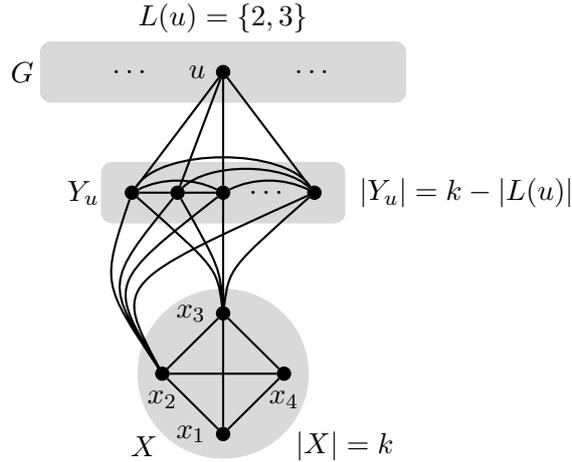
\begin{figure}[tb]
	\center
	\begin{tikzpicture}[scale=0.8]
	\tikzstyle{vertex}=[draw, circle, fill=black, minimum size=5pt, inner sep=0pt]
	\tikzstyle{set}=[draw,fill,black!15,rounded corners]
	
	\draw[set] (0,1) circle (1.4cm); \node at (-1.3,-0.2) {$X$}; \node at (2,-0.2) {$|X| = k$};
	\node[vertex,label=left:$x_1$] (x1) at (0,0) {};
	\node[vertex,label=below:$x_2$] (x2) at (-1,1) {};
	\node[vertex,label=left:$x_3$] (x3) at (0,2) {};
	\node[vertex,label=below:$x_4$] (x4) at (1,1) {};
	
	\draw[set] (-2,3.5) rectangle (2,4.5); \node at (-2.3,4) {$Y_u$}; \node at (4,4) {$|Y_u| = k-|L(u)|$};
	\node[vertex] (y1) at (-1.5,4) {};
	\node[vertex] (y2) at (-0.75,4) {};
	\node[vertex] (y3) at (0,4) {};
	\node[vertex] (y4) at (1.5,4) {};
	\node at (0.75,4) {$\dots$};
	
	
	\draw[set] (-3,5.5) rectangle (3,6.5); \node at (-3.3,6) {$G$};
	\node[vertex,label=left:$u$] (u) at (0,6) {};
	\node at (0,6.9) {$L(u)=\{2,3\}$};
	\node at (-1.5,6) {$\dots$};
	\node at (1.5,6) {$\dots$};
	
	
%
	
	\draw[thick] (x1)--(x2)--(x3)--(x4)--(x1)--(x3) (x2)--(x4);
	\draw[thick] (y1)--(y2)--(y3) .. controls +(0.5,0.25) and +(-0.5,0.25) .. (y4) (y1) .. controls +(0.5,0.25) and +(-0.5,0.25) .. (y3) (y2) .. controls +(0.5,0.5) and +(-0.5,0.5) .. (y4) (y1) .. controls +(0.5,0.75) and +(-0.5,0.75) .. (y4);
	\draw[thick] (u)--(y1) (u)--(y2) (u)--(y3) (u)--(y4);
	\draw[thick] (x2) .. controls +(-1,1.5) .. (y1) (x2) .. controls +(-0.95,1.5) .. (y2) (x2) .. controls +(-0.9,1.5) .. (y3) (x2) .. controls +(-0.8,1.5) .. (y4);
	\draw[thick] (x3) .. controls +(-0.1,0.75) .. (y1) (x3) .. controls +(-0.05,0.75) .. (y2) (x3) .. controls +(0,0.75) .. (y3) (x3) .. controls +(0.1,0.75) .. (y4);
	
	\end{tikzpicture}
	\caption{\label{fig:hard} Illustration of the NP-hardness construction for $k=4$.}
\end{figure} 

We now prove two claims that together with \NP-membership of {\sc Colouring} imply the theorem.

\begin{claimx}\label{cla:iff}
The graph $G$ has a colouring that respects $L$ if and only if $G'$ has a $k$-colouring.
\end{claimx}

\noindent {\bf Proof of \autoref{cla:iff}.}
First suppose that $G$ has a colouring $c$ that respects $L$. We give each~$x_i$ colour~$i$. Consider a clique $Y_u$.
By construction, every vertex in every $Y_u$ can only be assigned a colour from $\{1,\ldots,k\}\setminus L(u)$. 
As  $|Y_u|=k-|L(u)|$ and $Y_u$ is a clique, we need all of these colours. 
As every vertex in $Y_u$ has $u$ as its only neighbour in $G$ and $c(u)$ belongs to $L$, this is possible. 
Hence, we can extend~$c$ to a $k$-colouring~$c'$ of $G'$. 
Now suppose that $G'$ has a $k$-colouring~$c'$. As $X$ is a clique, 
we may assume without loss of generality that $c'(x_i)=i$ for $i=1,\ldots,k$. 
Then $c$ colours the vertices of each~$Y_u$ with colours from $\{1,\ldots,k\}\setminus L(u)$. 
As $Y_u$ is a clique of size $k-|L(u)|$, all these colours appear as a colour of a vertex in $Y_u$. 
This means that $u$ must get a colour from $L(u)$. Hence
 the restriction of $c'$ to $G$ yields a colouring~$c$ that respects $L$. 
This completes the proof of~\autoref{cla:iff}. \e

\begin{claimx}\label{cla:c4free}
The graph $G'$ is $(C_4,3P_3,P_3+P_6,2P_5,P_9)$-free.
\end{claimx}
\noindent {\bf Proof of \autoref{cla:c4free}.}
We first prove that $G'$ is $C_4$-free.
For contradiction, suppose that $G'$ contains an induced~$C_4$ on vertices $a_1,a_2,a_3,a_4$ in that order.
Then at least one of $a_1,a_2,a_3,a_4$, say $a_1$, is of $y$-type, since the $y$-type vertices separate $K$ from $G$, and both $K$ and $G$
are $C_4$-free. If $a_2$ and $a_4$ both belong to $X\cup Y_u$ or both $\{u\}\cup Y_u$, then they are adjacent, which is not possible. 
Hence, one of them, say $a_2$, belongs to $X$ and the other one, $a_4$, is equal to $u$. 
This is not possible either, as in that case $a_3$ must be of $y$-type and any two $y$-type neighbours of a $u$-type vertex are adjacent.  
We conclude that $G'$ is $C_4$-free.
 
We now prove that $G'$ is $3P_3$-free. For contradiction, suppose that $G'$ contains an induced~$3P_3$. 
At most one of the three connected components of the induced $3P_3$ can contain an $x$-type vertex. 
Then the other two connected components contain no $x$-type vertex. 
As the $y$-type neighbours of an $y$-type vertex form a clique together with their neighbour of $G$, 
both these other connected components contain at least two vertices from $G$. 
Then all these vertices must be of $i$-type, as $c$-type vertices are adjacent to every other vertex of $G$.
However, $i$-type vertices form an independent set and only share $c$-type vertices as common neighbours, a contradiction.
We conclude that $G'$ is $3P_3$-free.

We now prove that $G'$ is $(P_3+P_6)$-free. For contradiction, suppose that $G'$ contains an induced~$P_3+P_6$. Let $F_1$
be the $P_3$-component and $F_2$ be the $P_6$-component. Suppose $F_1$ contains an $x$-type vertex. 
Then $F_2$ contains no $x$-type vertex. As every  $\{u\}\cup Y_u$ is a clique and $i$-type vertices form an independent set, $F_2$
must contain at least one $c$-type vertex. As $c$-type vertices are adjacent to all $i$-type vertices and all other $c$-type vertices,
$F_2$ contains at most three vertices from $G$ which form a subpath of $F_2$. As $F_2$ contains six vertices and every
$\{u\}\cup Y_u$ is a clique, this is not possible. Hence $F_1$ contains no $x$-type vertex. As every $\{u\}\cup Y_u$ is a clique, 
this means that $F_1$ contains at least two adjacent vertices of $G$. As $i$-type vertices form an independent set, 
one of these vertices is of $c$-type. This means that $F_2$ contains no vertices of $G$. 
This is not possible as the $x$-type and $y$-type vertices induce a $P_5$-free graph.  We conclude that $G'$ is $(P_3+P_6)$-free.

We now prove that $G'$ is $2P_5$-free. For contradiction, suppose that $G'$ contains an induced~$2P_5$ with connected components
$F_1$ and $F_2$. At most one of $F_1$, $F_2$ may contain an $x$-type vertex. 
Hence we may assume that $F_2$ contains no $x$-vertex. This means that $F_2$ must be of the form $y-i-c-i-y$. 
As a consequence, $F_1$ only contains vertices of $x$-type or $y$-type. This is not possible as those vertices induce a $P_5$-free subgraph. 
We conclude that $G'$ is $2P_5$-free.

Finally we prove that $G'$ is $P_9$-free.
Let $P$ be a maximal induced path of $G'$. 
First  suppose that $P$ contains at least two $i$-type vertices. Then $P$ contains a subpath of the form $i-c-i$ or 
$i-y-x-x-y-i$.  We can extend $i-c-i$ to at most an 8-vertex path, which is of the form $y-i-c-i-y-x-x-y$, 
but we cannot extend $i-y-x-x-y-i$ any further. Now suppose that $P$ contains exactly one $i$-type vertex. 
If $P$ contains no $c$-type vertex, then
$P$ is a 5-vertex path of the form $i-y-k-k-y$. Otherwise $P$ contains exactly one $c$-type vertex. 
In that case $P$ is a 7-vertex path of the form $y-c-i-y-x-x-y$ or $y-x-x-y-c-i-y$. 
Now suppose that $P$ has no $i$-type vertex. If $P$ has no $c$-type vertex either,
then $P$ is of the form $y-x-y$ or $y-x-x-y$, so $P$ has at most five vertices. If $P$ has exactly one $c$-type vertex,
then $P$ is a 5-vertex path of the form $c-y-k-k-y$. Otherwise, $P$ has exactly two $c$-type vertices. 
In that case $P$ is a 7-vertex path of the form $y-c-c-y-k-k-y$. We conclude that $G'$ is $P_9$-free. \e

\medskip
\noindent
The result now follows from \autoref{cla:iff} and \autoref{cla:c4free}.
\end{proof}

\section{Conclusions}\label{s-con}

We proved that \cn restricted to $(C_4,P_t)$-free graphs is polynomial-time solvable for $t\leq 6$ and \NP-complete for $t\geq 9$.
Combined with the aforementioned known results from~\cite{BKM06,HJP14,KKTW01}, we can replace \autoref{t-5} by the following almost complete dichotomy for \cn restricted to $(C_s,P_t)$-free graphs.

\begin{theorem}\label{t-almostclass}
Let $s\geq 3$ and $t\geq 1$ be two fixed integers. 
Then {\sc Colouring} for $(C_s,P_t)$-free graphs 
is polynomial-time solvable if $s=3$, $t\leq 6$, or $s=4$, $t\leq 6$, or $s\geq 5$, $t\leq 4$,
and \NP-complete if $s=3$, $t\geq 22$, or $s=4$, $t\geq 9$, or $s\geq 5$, $t\geq 5$.
\end{theorem}

We proved that, in contrast to $(C_4,P_6)$-free graphs, $(C_4,P_6)$-free atoms have bounded clique-width. 
As we also showed that the classification of boundedness of clique-width of $H$-free graphs and $H$-free atoms coincides,
this result was not expected beforehand. As such,
we believe that a systematic study in the applicability of this technique, together with the other techniques developed in our paper, can be used to prove further polynomial-time results for \cn. For future work we aim to complete the classification of \autoref{t-almostclass}. 

The natural candidate class for a polynomial-time result of \cn is
the class of $(C_4,P_7)$-free graphs. However, this may require significant efforts for the following reason.  
Lozin and Malyshev~\cite{LM15} determined the complexity of {\sc Colouring} for ${\cal H}$-free graphs 
for every finite set of graphs~${\cal H}$ consisting only of 4-vertex graphs except when 
${\cal H}$ is $\{K_{1,3},4P_1\}$, $\{K_{1,3},2P_1+P_2\}$, $\{K_{1,3},2P_1+P_2,4P_1\}$ or $\{C_4,4P_1\}$.
Solving any of these open cases would be considered as a major advancement in the area.
Since $(C_4,4P_1)$-free graphs are $(C_4,P_7)$-free,
polynomial-time solvability of {\sc Colouring} on $(C_4,P_7)$-free graphs 
implies polynomial-time solvability for {\sc Colouring} on $(C_4,4P_1)$-free graphs.
As a first step, we aim to apply the techniques of this paper to $(C_4,4P_1)$-free graphs.

The class of $(C_3,P_7)$-free graphs is also a natural class to consider.
Interestingly, every $(C_3,P_7)$-free graph is 5-colourable. This follows from a result of Gravier, Ho\`ang and Maffray~\cite{GHM03} 
who proved that for any two integers $r,t\geq 1$, every $(K_r,P_t)$-free graph can be coloured with at most $(t-2)^{r-2}$ colours.
On the other hand,  $3$-{\sc Colouring} is polynomial-time solvable for $P_7$-free graphs~\cite{BCMSZ}. 
Hence, in order to solve {\sc Colouring} for $(C_3,P_7)$-free graphs 
we may instead consider $4$-{\sc Colouring} for $(C_3,P_7)$-free graphs. This problem also seems highly nontrivial.

\paragraph*{Acknowledgments.}
Initially we proved NP-hardness of {\sc Colouring} for $(C_4,P_{16})$-free graphs. 
Afterwards we were able to improve this result to $(C_4,P_9)$-free graphs via a simplification of our construction. 
We would like to thank an anonymous reviewer of the conference version of our paper for pointing out this simplification as well.
Serge Gaspers is the recipient of an Australian Research Council (ARC) Future Fellowship (FT140100048)
and acknowledges support under the ARC's Discovery Projects funding scheme (DP150101134).
Dani\"el Paulusma is supported by Leverhulme Trust Research Project Grant RPG-2016-258.

\appendix

\section{The Proof of Lemma~\ref{lem:co-bipartite chain graph}}\label{a-a}

Here is a proof for Lemma~\ref{lem:co-bipartite chain graph}, which we restate below.

\medskip
\noindent
{\bf Lemma~\ref{lem:co-bipartite chain graph} [Folklore].}
{\it There is a nice $4$-expression for any co-bipartite chain graph.}

\begin{proof}
Let $G=(A \uplus B,E)$ be a co-bipartite chain graph where $A$ and $B$ are cliques.
Since $G$ is a co-bipartite chain graph, we can order the vertices in $A$ as $a_0, a_1,\ldots,a_s$ 
and the vertices in $B$ as $b_1,\ldots,b_t$ such that for each $0\le i\le s$,
$N_B(a_i)=\{b_1,\ldots, b_j\}$ for some $0\le j\le t$
($j=0$ means that $N_B(a_i)=\emptyset$)
and $N_B(a_0)\subseteq N_B(a_1)\subseteq \ldots \subseteq N_B(a_t)$. 
Note that two vertices in $A$ are twins in $G$ if and only if they have the same neighbours in $B$.
It follows from \autoref{lem:prime} that twin vertices do not change the clique-width. Neither
do they change the niceness of the clique-width expression. Therefore,
we may assume that for each  $0\le j\le t$ there is at most one $a_i$ with $N_B(a_i)=\{b_1,\ldots, b_j\}$.
Moreover, we can assume that for each  $0\le j\le t$ there is exactly one $a_i$ with $N_B(a_i)=\{b_1,\ldots, b_j\}$
for otherwise $G$ would be an induced subgraph of this graph. In other words, $s=t+1$ and $N_B(a_i)=\{b_1,\ldots,b_i\}$
for each $0\le i\le t$.
Let $\tau_1=\rho_{3\rightarrow 4}(\rho_{1\rightarrow 2}(\eta_{1,3}(1(a_1)\oplus 3(b_1))))$.
For each $2\le i\le t$, note that
$$\tau_i=\rho_{1\rightarrow 2}(\eta_{1,4}(\rho_{3\rightarrow 4}(\eta_{3,4}(\eta_{1,2}((1(a_i)\oplus 3(b_i))\oplus \tau_{i-1})))))$$
is a nice $4$-expression for $G[\{a_1,\ldots,a_i,b_1,\ldots,b_i\}]$.
Now $\tau=\rho_{1\rightarrow 2}(\eta_{1,2}(1(a_0)\oplus \tau_t))$ is a nice $4$-expression for $G$.
\end{proof}

\end{document}